\documentclass{amsart}

\usepackage{amsmath,amsfonts,amssymb}
\usepackage{mathrsfs}
\usepackage{longtable}

\usepackage{xcolor}
\usepackage{lunadiagrams}

\newtheorem{theorem}{Theorem}[section]
\newtheorem{lemma}{Lemma}[section]
\newtheorem{proposition}{Proposition}[section]
\newtheorem{corollary}{Corollary}[section]
\newtheorem{conjecture}{Conjecture}[section]

\theoremstyle{definition}
\newtheorem{definition}{Definition}[section]
\newtheorem{remark}{Remark}[section]

\numberwithin{equation}{section}

\newcommand{\cone}{\Span_{\QQ_{\geq0}}}

\newcommand{\rank}{\operatorname{rank}}

\newcommand{\CC}{\mathbb C}
\newcommand{\PP}{\mathbb P}

\newcommand{\ZZ}{\mathbb Z}

\newcommand{\QQ}{\mathbb Q}

\newcommand{\calC}{\mathcal C}
\newcommand{\calD}{\mathcal D}
\newcommand{\calM}{\mathcal M}
\newcommand{\calN}{\mathcal N}
\newcommand{\calV}{\mathcal V}
\newcommand{\calU}{\mathcal U}
\newcommand{\scrP}{\mathscr P}

\newcommand{\N}{\mathrm N}
\newcommand{\GL}{\mathrm{GL}}
\newcommand{\SL}{\mathrm{SL}}
\newcommand{\SO}{\mathrm{SO}}
\newcommand{\Sp}{\mathrm{Sp}}

\newcommand{\Id}{\mathrm{Id}}
\newcommand{\Span}{\mathrm{span}}

\newcommand{\<}{\langle}
\renewcommand{\>}{\rangle}

\begin{document}

\title[The generalized Mukai conjecture]{The generalized Mukai conjecture for spherical varieties with a reductive general isotropy group}
\author{P.\ Bravi and G.\ Pezzini}
\address{Dipartimento di Matematica Guido Castelnuovo\\
Sapienza Universit\`a di Roma\\
Piazzale Aldo Moro 5\\
00185, Roma\\
Italy}
\email{bravi@mat.uniroma1.it}
\email{pezzini@mat.uniroma1.it}

\maketitle

\begin{abstract}
In the case of spherical varieties with reductive general isotropy group we prove a conjecture of G.~Gagliardi and J.~Hofscheier, which implies the generalized Mukai conjecture of L.~Bonavero, C.~Casagrande, O.~Debarre and S.~Druel for these varieties.
\end{abstract}

\section{Introduction}

The generalized Mukai conjecture by Bonavero, Casagrande, Debarre and Druel is the following.

\begin{conjecture}\cite{BCDD}\label{GMC}
Let $X$ be a smooth Fano variety. Let $\rho_X$ denote the Picard number of $X$ and $\iota_X$ its \emph{pseudo-index}, that is,
\[\iota_X=\min\{-K_X\cdot C\ :\ C\mbox{ is a rational curve in }X\}\]
where $-K_X$ is the anticanonical divisor of $X$.
Then 
\[\rho_X(\iota_X-1)\leq\dim X\]
and equality holds if and only if $X$ is isomorphic to $(\mathbb P^{\iota_X-1})^{\rho_X}$.
\end{conjecture} 

It has been analyzed and proven in many cases, see \cite{ACO, Ca, P, GH, F, CPS, R} among others.

Gagliardi and Hofscheier considered the case where $X$ is spherical and made the following, which we will call the Gagliardi-Hofscheier conjecture (for all notations we refer to Section~\ref{s:GHconj}).

\begin{conjecture}\cite{GH}\label{GHC}
Let $X$ be a complete spherical variety. Let $\mathscr P(X)$ denote the supremum of
\[\sum_{D\in\mathcal D}(m_D-1+\langle\rho'(D),\theta\rangle)\]
for $\theta\in\mathrm{cone}(\Sigma)$ such that $\langle\rho'(D),\theta\rangle\geq-m_D$ for all $D\in\mathcal D$,
where $\mathcal D$ is the set of $B$-invariant prime divisors of $X$ and $-K_X=\sum_{D\in\mathcal D}m_D D$. 
Then 
\[\mathscr P(X)\leq \dim X - \rank X\]
and equality holds if and only if $X$ is isomorphic to a toric variety.
\end{conjecture}

Moreover, they proved the following.

\begin{theorem}\cite{GH}
Let $X$ be a smooth spherical Fano variety. If Conjecture \ref{GHC} holds for $X$ then Conjecture \ref{GMC} also holds for $X$. 
\end{theorem}

Finally, in the same paper, they proved Conjecture \ref{GHC} for the symmetric varieties. 

With the eventual goal of proving the full Gagliardi-Hofscheier conjecture using the classification of spherical verieties (see \cite{LV,K,Lu,Lo,BP} and the references therein), in this paper we verify the conjecture in the case of spherical varieties with reductive general isotropy group. 

\begin{theorem}\label{MT}
The Gagliardi-Hofscheier conjecture holds for spherical varieties with reductive general isotropy group.
\end{theorem}

In Section~\ref{s:GHconj} and Section~\ref{s:skeleton} we recall the basic definitions and the reduction of the Gagliardi-Hofscheier conjecture into combinatorial terms. In Section~\ref{s:modules} we provide a conceptual proof in the case of sperical modules. In Section~\ref{s:proof} we prove Theorem~\ref{MT} with a case-by-case analysis.

\subsubsection*{Acknowledgments} We would like to thank G.~Gagliardi and J.~Hofscheier for sharing with us their ideas on the topic at an early stage of the present work.

\section{Combinatorics of spherical varieties}\label{s:GHconj}

Let us recall here the basic definitions, following mostly the notations of \cite{GH}.

Let $G$ be a connected reductive complex algebraic group, $T$ a maximal torus in $G$ and $B$ a Borel subgroup of $G$ containing $T$. Let $R$ be the root system associated with $G$ and $T$, let $S\subset R$ be the set of simple roots associated with $B$ and $R^+$ the corresponding set of positive roots. 

Let $X$ be a {\em spherical} $G$-variety, that is, a normal irreducible algebraic variety endowed with an algebraic action of $G$ with an open $B$-orbit.  

Let $\mathcal M$ denote {\em weight lattice} of $X$, i.e.\ the lattice of weights of $B$-semi-invariant rational functions on $X$. The rank of $\mathcal M$ is called the {\em rank} of $X$. Consider also its dual $\mathcal N:=\mathrm{Hom}(\mathcal M, \mathbb Z)$ and the natural pairing $\langle\ ,\ \rangle\colon\mathcal N\times\mathcal M\to\mathbb Z$.

Let $\mathcal D$ denote the set of $B$-invariant prime divisors of $X$. For any $D\in\mathcal D$ one has a functional $\rho'(D)$ on $\mathcal M$ defined as $\langle\rho'(D),\chi\rangle=\nu_D(f_\chi)$, the order of vanishing on $D$ of the (uniquely determined up to scalar factor) $B$-semi-invariant rational function $f_\chi$ on $X$ of weight $\chi$.

The set $\mathcal D$ can be partitioned into two subsets: the set $\Gamma$ of $G$-invariant prime divisors and the set $\Delta$ of non-$G$-invariant ones, called {\em colors}, which are the closures of the $B$-invariant prime divisors in the open $G$-orbit $Y$ of $X$. 

For every simple root $\alpha$, let $\Delta(\alpha)$ denote the set of colors which are not invariant under the action of the minimal parabolic subgroup containing $B$ and corresponding to $\alpha$. The cardinality of $\Delta(\alpha)$ is at most two. 

Let $S^p\subset S$ denote the set of simple roots $\alpha$ such that $\Delta(\alpha)=\emptyset$. This coincides with the set of simple roots corresponding to the parabolic subgroup of $G$ stabilizing the open $B$-orbit. 

Let us denote by $H$ the stabilizer of a point in the open $G$-orbit $Y$. Then $H$ is called a {\em spherical subgroup} of $G$, and the normalizer $N_G(H)$ of $H$ in $G$ acts naturally by $G$-equivariant automorphisms of $Y$. Therefore $N_G(H)$ stabilizes the open $B$-orbit of $Y$ and acts by permutation of the set $\Delta$ of its colors, the subgroup $H$ acting trivially. The kernel of the permutation action is called the \emph{spherical closure} $\overline H$ of $H$. Clearly $\overline H$ is spherical too and one has $\overline{\overline H}=\overline H$. A spherical subgroup $H$ is called \emph{spherically closed} if $\overline H = H$.

All $G$-invariant discrete valuations on the field of rational functions on $X$ induce elements of $\mathcal N_{\mathbb Q}$ in the same way as $\nu_D$ above. These elements form a convex polyhedral cone $\mathcal V\subset \mathcal N_{\mathbb Q}$, i.e.\ there exists a minimal finite subset $\Sigma$ of $\mathcal M$ such that
\[
\mathcal V=\{\nu\in\mathcal N_{\mathbb Q}\ :\ \langle \nu,\gamma\rangle\leq 0,\ \forall\, \gamma\in\Sigma\}.
\] 
The elements of $\Sigma$ are uniquely determined up to normalization: if they are required to be primitive in $\mathcal M$ then they are called the \emph{spherical roots} of $X$. In this paper we use a common and slightly different normalization, namely $\Sigma$ will denote the set of spherical roots of the homogeneous space $G/\overline H$.

With this choice the elements of $\Sigma$ are called \emph{spherically closed spherical roots} of $X$. They are known to be linearly independent (that is, the cone $\mathcal V$ is cosimplicial) and of special kind. In fact for any reductive group $G$ there is a finite list of possible elements of the root lattice, actually positive roots or sum of two positive roots, which can be spherically closed spherical roots of a spherical $G$-variety, see \emph{loc.cit.}  

The last notation we introduce for $X$ regards an anticanonical divisor, which is obtained, by \cite{B}, as
\[-K_X=\sum_{D\in\mathcal D}m_DD,\]
where the $m_D$'s are positive integers (well-)defined as follows:
\begin{itemize}
\item[-] if $D$ is $G$-invariant then $m_D=1$,
\item[-] if $D\in\Delta(\alpha)$ for a simple root $\alpha$ belonging to $\Sigma$ or $\frac12\Sigma$, then $m_D=1$,
\item[-] if $D\in\Delta(\alpha)$ for a simple root $\alpha$ belonging neither to $\Sigma$ nor to $\frac12\Sigma$, then  $m_D=\langle\alpha^\vee,2\rho_S-2\rho_{S^p}\rangle$, where $2\rho_S$ and $2\rho_{S^p}$ denote the sum of the positive roots supported respectively in $S$ and $S^p$.
\end{itemize} 

We will consider polytopes $Q\subset\mathcal N_{\mathbb Q}$ and their duals $Q^*\subset\mathcal M_{\mathbb Q}$ defined by the inequalities $\langle \nu, \mu \rangle \geq -1$, for all $\nu\in Q$. A vertex $\mu$ of $Q^*$ is called \emph{supported} if the intersection of $\mu+ \mathrm{cone}(\Sigma)$ with $Q^*$ equals $\{\mu\}$.

Since all objects defined so far essentially only depend on the open $G$-orbit $Y$, we can give the following definition.

\begin{definition}\cite{P, GH15}
A polytope $Q\subset \mathcal N_{\mathbb Q}$ is called $Y$-reflexive if:
\begin{itemize}
\item[-] for all $D\in\Delta$, $\rho'(D)/m_D\in Q$,
\item[-] the origin $0\in\mathcal N$ is contained in the interior of $Q$,
\item[-] every vertex of $Q$ belongs to $\{\rho'(D)/m_D\ :\ D\in\Delta\}$ or to $\mathcal N\cap \mathcal V$,
\item[-] every supported vertex of $Q^*$ belongs to $\mathcal M$.
 \end{itemize}
\end{definition}

One can see that if $X$ is a Gorenstein spherical Fano $G$-variety with open $G$-orbit $Y$ then the convex hull of $\{\rho'(D)/m_D\ :\ D\in\mathcal D\}$ is a $Y$-reflexive polytope and moreover:

\begin{theorem}\cite{GH15}
The correspondence $X\mapsto\mathrm{conv}\{\rho'(D)/m_D\ :\ D\in\mathcal D\}$ gives a bijection between Gorenstein spherical Fano $G$-varieties with open $G$-orbit $Y$, up to $G$-equivariant isomorphism, and $Y$-reflexive polytopes.  
\end{theorem}

Notice that the numerical invariant $\mathscr P(X)$ defined in the introduction is the supremum of the functional
\[\sum_{D\in\mathcal D}(m_D-1+\langle\rho'(D),\theta\rangle)\]
for $\theta$ lying in the intersection of $\mathrm{cone}(\Sigma)$ and the dual of $Q_X:=\mathrm{conv}\{\rho'(D)/m_D\ :\ D\in\mathcal D\}$.

It is easy to see that $\mathscr P(X)\geq 0$, and finite if $X$ is complete. 

\section{The spherical skeleton}\label{s:skeleton}

In order to reduce its proof to a case-by-case analysis, Gagliardi and Hofscheier have translated Conjecture \ref{GHC} in purely combinatorial terms, as follows.

For a spherically closed spherical subgroup $H\subset G$ the datum of $\Delta$, $S^p$ and $\Sigma$  associated with $G/H$ uniquely determines the subgroup $H$ up to $G$-conjugation, or equivalently the homogeneous spherical variety $G/H$, up to $G$-equivariant isomorphism. Moreover, there is an axiomatic definition of \emph{spherically closed spherical system} for $G$ such that the correspondence
\begin{equation}\label{corr2}
H\mapsto  (\Delta, S^p, \Sigma)
\end{equation}
gives a bijection between $G$-conjugacy classes of spherically closed spherical subgroups of $G$ and spherically closed spherical systems for $G$, \cite{Lu, Lo, BP}. 

We are ready to define the main combinatorial object used in this paper, the spherical skeleton, first introduced by R.\ Camus in \cite{C}. Let $X$ be a spherical variety and $(\Delta,S^p,\Sigma)$ the spherically closed spherical system associated with the spherical closure of a principal isotropy group of $X$. We add to this datum a fourth component $\Gamma$ which is the set of $G$-stable prime divisors of $X$, implicitly endowed with the pairing $\rho\colon \Gamma\times \Sigma \to \mathbb Z_{\leq0}$. The quadruple $\mathcal R_X=(\Delta,S^p,\Sigma, \Gamma)$ is called the {\em spherical skeleton} of $X$. The purely combinatorial notion is the following.

\begin{definition}\cite{GH}
A quadruple $\mathcal R =(\Delta, S^p, \Sigma, \Gamma)$, where $(\Delta, S^p, \Sigma)$ is a spherically closed spherical system for $G$ and $\Gamma$ is a finite set endowed with a pairing $\rho\colon \Gamma\times\Sigma\to\mathbb Z_{\leq 0}$, is called a \emph{spherical skeleton} for $G$.  The union $\Delta\cup\Gamma$ is denoted by $\mathcal D$, and the spherical skeleton $\mathcal R=(\Delta, S^p, \Sigma, \Gamma)$ is called \emph{complete} if the cone generated by the $\rho(D)$'s, for $D\in\mathcal D$, covers the whole $(\Span_\QQ\Sigma)^*$.  
\end{definition}

With a spherical skeleton $\mathcal R=(\Delta,S^p,\Sigma,\Gamma)$ we can associate a polytope 
\[Q_{\mathcal R}:=\mathrm{conv}\{\rho(D)/m_D\ :\ D\in\mathcal D\},\]
where $m_D:=1$ for all $D\in\Gamma$, its dual $Q^*_{\mathcal R}\subset(\mathrm{span}\,\Sigma)^*_{\mathbb Q}$, and define $\mathscr P(\mathcal R)$ to be the supremum of 
\[\sum_{D\in\mathcal D}(m_D-1+\langle\rho(D),\theta\rangle)\]
for $\theta\in Q^*_{\mathcal R}\cap\mathrm{cone}(\Sigma)$.

For all spherical $G$-varieties $X$ we have $Q^*_{\mathcal R_X}=Q^*_X\cap\mathrm{span}_{\mathbb Q}\Sigma$ and $\mathscr P(\mathcal R_X)=\mathscr P(X)$, moreover the number $\dim(X)-\rank(X)$ appearing in Conjecture~\ref{GHC} is equal to $|R^+| - |R^+_{S^p}|$, where $R^+_{S^p}$ is the set of positive roots that are in the linear span of $S^p$.

Gagliardi and Hofscheier reformulated their Conjecture~\ref{GHC} also in terms of spherical skeletons and {\em spherical modules}, which are defined as $G$-modules that are also spherical as $G$-varieties. These modules have been classified \cite{Ka, BR, Le} and their spherical skeletons are known \cite{G}. In particular, although a spherical module is not complete, its spherical skeleton turs out to be always complete (see Corollary~\ref{cor:skeletonV}).

In general, two spherical skeletons $\mathcal R_i=(\Delta_i,S^p_i,\Sigma_i,\Gamma_i)$, $i=1,2$, are said to be isomophic if there is an isomorphism between the two underlying root systems $R_1$ and $R_2$, inducing bijections between $S^p_1$ and $S^p_2$ and between $\Sigma_1$ and $\Sigma_2$, together with a compatible bijection between $\mathcal D_1$ and $\mathcal D_2$.  

Gagliardi and Hofscheier proved that Conjecture \ref{GHC} is equivalent to the following

\begin{conjecture}\label{GHC2}\cite{GH}
Let $\mathcal R=(\Delta, S^p,\Sigma,\Gamma)$ be a complete spherical skeleton, then 
\begin{equation}\label{ineq}
\mathscr P(\mathcal R)\leq |R^+| - |R^+_{S^p}|
\end{equation}
and equality holds if and only if $\mathcal R$ is isomorphic to the spherical skeleton $\mathcal R_V$ of a spherical module $V$.  
\end{conjecture}

\section{Spherical modules}\label{s:modules}

In this section we verify the formula~(\ref{ineq}) in the special case of spherical module $X$, i.e.\ we prove the equality $\mathscr P(X)=\dim(X)-\rank(X)$. We give a conceptual proof, using some considerations on the so-called {\em weight monoid} of $X$ that do not seem to be available in the literature. However it is also possible to verify the equality directly case-by-case using the classification of spherical modules.

The weight monoid of $X$ is the subset $\calM^+\subset\calM$ of the weights of $B$-eigenvectors in the ring of regular functions $\CC[X]$.

For any $\theta$ in the $\QQ$-span of $\Sigma$ we set
\[
p_X(\theta) = \sum_{D\in\calD} \left(m_D-1+\<\rho(D),\theta\>\right).
\]
For any $G$-orbit $Z\subset X$, we denote by $(\calC_Z, \Delta_Z)$ the {\em colored cone} corresponding to $Z$, i.e.\ $\calC_Z$ is the convex cone generated by $\rho'(D)$ for all $D\in \calD$ such that $D\supset Z$, and $\Delta_Z$ is the set of all $D\in\Delta$ containing $Z$.

\begin{proposition}\label{prop:mfs-geq}
Let $X$ be a spherical $G$-module. Then we have $\mathscr P(X)\geq \dim(X)-\rank(X)$.
\end{proposition}
\begin{proof}
We intend to find a point $\vartheta$ where the value $p_X(\vartheta)=\dim(X)-\rank(X)$ is attained.

Decompose $X=V_1\oplus\ldots\oplus V_t$ into a sum of irreducible $G$-modules, and consider the $G$-equivariant completion
\[
\overline X = \PP(V_1\oplus \CC)\times \ldots\times \PP(V_t\oplus \CC) = (V_1\cup\PP(V_1))\times \ldots\times (V_t\cup\PP(V_t))
\]
of $X$, where $G$ acts trivially on the summands $\CC$. The difference $\overline X\smallsetminus X$ is the union of $G$-stable prime divisors $E_1,\ldots,E_t$, where $E_i$ is the product of $\PP(V_i)$ and of the factors $\PP(V_j\oplus\CC)$ for all $j\neq i$.

Let us describe the $B$-stable prime divisors of $\overline X$ that intersect $X$: we have the closures $D_{i,1},\ldots,D_{i,n(i)}$ of the inverse images in $X$ of the $B$-stable prime divisors of $V_i$ under the projection, and additional prime divisors $D_{t+1,1},\ldots,D_{t+1,n(t+1)}$ not obtained in this way from any projection. The weight monoid of $X$ is freely generated by the $B$-eigenvalues $\lambda_{i,s}$ of global equations $f_{i,s}$ of $D_{i,s}$ for all $i\in\{1,\ldots,t+1\}$ and $s\in\{1,\ldots,n(i)\}$. The valuations $\rho'(D_{i,s})$ for all $i,s$ as above form the dual basis of these generators. We may assume that $D_{i,1}$ corresponds to the $B$-stable linear hyperplane of $V_i$ for all $i\in\{1,\ldots,t\}$, i.e.\ that $f_{i,1}$ is a functional on $V_i$ and $\lambda_{i,1}$ is the highest weight of the dual of $V_i$.

For all $j,i,s$ the value $e_{j,i,s}=\<\rho'(E_j),\lambda_{i,s}\>$ is equal to the opposite of the partial degree $\deg_j(f_{i,s})$ of $f_{i,s}$ with respect to $V_j$, so this value is $\leq 0$. We record the easy fact that $e_{j,i,1}=-\deg_j(f_{i,1})=-\delta_{i,j}$ for all $i,j\in\{1,\ldots,t\}$.

Among the closed $G$-orbits of $\overline X$ we find $Z_0=\{0\}\subset X$. The convex cone $\calC_{Z_0}$ is generated by $\rho'(D_{i,s})$ for all $i,s$ as above. Any other closed $G$-orbit $Z\neq Z_0$ of $\overline X$ is contained in the union $\partial X = E_1\cup \ldots\cup E_t$. In this case there exists some $j\in\{1,\ldots,t\}$ such that $Z\subset E_j$, and we claim that for those indices $j$ the prime divisor $D_{j,1}$ doesn't contain $Z$. The claim follows by observing that $Z$ projects surjectively onto the unique closed $G$-orbit $W_{j}$ of $\PP(V_j)$, and $D_{j,1}\cap X$ projects onto the $B$-stable linear hyperplane of $V_j$. This hyperplane cannot contain the cone over $W_j$ because this cone is $G$-stable hence spans linearly the whole $V_j$, whence the claim.

For any closed $G$-orbit $Z\subset \overline X$, call $J(Z)$ the set of indices $j$ such that $E_j$ contains $Z$. For later convenience let us define the following sets:
\[
A = \left\{ (i,s)\in(\ZZ_{\geq0})^2 \;|\; i\in\{1,\ldots,t+1\}, s\in\{1,\ldots,n(i)\}\right\},
\]
\[
N = A\smallsetminus (\{1,\ldots,t\}\times\{1\}),
\]
\[
K(Z) = A\smallsetminus (J(Z)\times\{1\}).
\]
As a consequence of our analysis, the convex cone $\calC_Z$ is generated by $\rho(E_j)$ for all $j\in J(Z)$, and some other generators of the form $\rho'(D_{i,s})$ for $(i,s)\in A$. Among these, the divisors $D_{j,1}$ for $j\in J(Z)$ do not contain $Z$, but the cone $\calC_Z$ has maximal dimension equal to $|A|$, hence the other generators of $\calC_Z$ must be $\rho'(D_{i,s})$ for all couples $(i,s)\in K(Z)$.

Recall that $\overline X$ is projective, therefore the union of all its colored cones contains $\calV$. The crucial claim of the proof is now that $\calV$ is also contained in the convex cone $\calU$ generated by $\pm\rho'(E_j)$ for all $j\in\{1,\ldots,t\}$, and by $\rho'(D_{i,s})$ for all $(i,s)\in N$.

Let us show the claim, so take $v\in\calV$. For any subset $J\subset \{1,\ldots,t\}$, set $J^c=\{1,\ldots,t\}\smallsetminus J$. The vectors $\rho'(E_j)$ for $j\in J$ together with $\rho'(D_{j,1})$ for $j\in J^c$ and with $\rho'(D_{i,s})$ for $(i,s)\in N$ form a basis of the vector space $\calN_\QQ$, because of the equality $e_{j,i,1}=-\delta_{i,j}$ for all $i,j\in\{1,\ldots,t\}$.

Write
\[
v = \sum_{(i,s)\in N}c_{i,s}\rho'(D_{i,s}) + \sum_{j=1}^t c_j\rho'(E_j)
\]
with coefficients $c_{i,s}, c_j\in\QQ$. Since $\rho'(E_j)\in\calV$, the sum
\[
w = \sum_{(i,s)\in N}c_{i,s}\rho'(D_{i,s}) + \sum_{j=1}^t |c_j|\rho'(E_j)
\]
is in $\calV$. We intend to prove that $w\in\calU$: this will imply that $v\in\calU$ because $v-w\in\calU$ and $\calU$ is a convex cone.

Let $Z\subset \overline X$ be a closed $G$-orbit such that $\calC_Z$ contains $w$, and write
\[
w = \sum_{(i,s)\in N}b_{i,s}\rho'(D_{i,s}) + \sum_{j\in J(Z)^c} b_j\rho'(D_{j,1}) + \sum_{j\in J(Z)}b_j\rho'(E_j)
\]
with uniquely determined coefficients $b_{i,s},b_j\in\QQ_{\geq0}$. The equality
\[
\rho'(E_j)=\sum_{(i,s)\in N} e_{j,i,s}\rho'(D_{i,s}) - \rho'(D_{j,1})
\]
used for the indices $j$ in $J(Z)^c$ yields
\[
w = \sum_{(i,s)\in N}\left(c_{i,s}+\sum_{j\in J(Z)^c}|c_j|e_{j,i,s}\right)\rho'(D_{i,s}) - \sum_{j\in J(Z)^c} |c_j|\rho'(D_{j,1})+\sum_{j\in J(Z)} |c_j|\rho'(E_j).
\]
We conclude that $|c_j|=-b_j\leq 0$ for all $j\in J(Z)^c$, that $|c_j|=b_j\geq 0$ for all $j\in J(Z)$, and that $c_{i,s}+\sum_{j\in J(Z)^c}|c_j|e_{j,i,s}=b_{i,s}\geq 0$ for all $(i,s)\in N$. But then $c_j=0$ for all $j\in J(Z)^c$, which yields $c_{i,s}\geq 0$ for all $(i,s)\in N$ and $w\in\calU$, proving the claim.

We exploit the claim to finish the proof of the proposition. An anticanonical divisor on $\overline X$ can be expressed in two ways:
\[
\sum_{i=1}^t(\dim(V_i) D_{i,1} + E_i), \qquad\text{and}\qquad \sum_{(i,s)\in A} m_{i,s}D_{i,s} + E_1+\ldots+E_t
\]
where $m_{i,s}=m_{D_{i,s}}$, therefore the divisor
\[
\delta=\sum_{i=1}^t(\dim(V_i) D_{i,1}) - \sum_{(i,s)\in A} m_{i,s}D_{i,s}
\]
of $\overline X$ is principal. Let $f\in\CC(\overline X)$ a $B$-semiinvariant function with divisor $\delta$, and let $\vartheta$ be its $B$-eigenvalue. Let us show that $\vartheta\in \cone \Sigma$.

Recall that $\vartheta\in \cone \Sigma$ if and only if all elements of $\calV$ take non-positive values on $\vartheta$. The above expression of $\delta$ shows that the valuation $\rho'(E_j)$ vanishes on $\vartheta$ for all $j\in\{1,\ldots,n\}$, and the valuations $\rho'(D_{i,s})$ for $(i,s)\in N$ take negative values on $\vartheta$ because $m_{i,s}>0$ for all $i,s$. This shows that all elements of $\calU$ take non-positive values on $\vartheta$, proving that $\vartheta\in  \cone\Sigma$.

We also check $\vartheta\in Q^*$. For this, we compute for all $D\in \Delta$ as follows: if $D=D_{i,s}$ for some $(i,s)\in N$ we have
\[
\< \vartheta, \rho'(D)\> = -m_{i,s},
\]
if $D=D_{i,1}$ for some $i\in\{1,\ldots,t\}$ we have
\[
\< \vartheta, \rho'(D)\> = \dim(V_i)\geq -m_{i,1}
\]
and finally if $D=E_i$ for some $i\in\{1,\ldots,t\}$ we have
\[
\< \vartheta, \rho'(D)\> = 0 \geq -1=-m_{E_{i}},
\]
hence $\vartheta\in Q^*\cap \cone\Sigma$. Finally we compute the desired value $p_X(\vartheta)$, recalling that $\calD$ is the set of all $B$-stable prime divisors of $X$:
\[
\sum_{D\in\calD} \left(m_D-1+\<\rho'(D),\vartheta\>\right) =
A_1 + A_2
\]
where
\[
A_1 = \sum_{(i,s)\in N} \left(m_{i,s}-1-m_{i,s}\right) = - |N|
\]
and
\[
A_2 = \sum_{i=1}^t \left( m_{i,1}-1+\dim(V_i)-m_{i,1}\right) = \dim(V)-t
\]
so
\[
A_1+A_2 = \dim(X) - \rank(X)
\]
because $|N|+t = |A|=\rank(X)$.
\end{proof}

We prepare now for the proof of the equality $\scrP(X)=\dim(X)-\rank(X)$ by analyzing the weight monoid $\calM^+$ of the spherical module $X$.

Decompose $X=V_1\oplus\ldots\oplus V_t$ into a sum of irreducible modules. For simplicity we assume that $G$ is contained in $\GL(X)$ and that it contains the group $\CC^*_i$ of homotheties of $V_i$ for all $i$. It is not difficult to check that this has no influence on our considerations on $\scrP$. Set $C=\CC^*_1\times\ldots\times\CC^*_t$.

We simplify the notations of the proof of Proposition~\ref{prop:mfs-geq}, and we denote simply by $\lambda_1,\ldots,\lambda_t,\lambda_{t+1},\ldots,\lambda_r$ the free generators of the monoid $\calM^+$, where $\lambda_i$ is the highest weight of the dual of $V_i$ for all $i\in\{1,\ldots,t\}$. Correspondingly we denote by $D_1,\ldots,D_r$ the $B$-stable prime divisors of $X$.

\begin{lemma}\label{lemma:rank}
We have $r \geq t+|\Sigma|$.
\end{lemma}
\begin{proof}
The number $r$ is the rank of $X$ as a spherical variety, and it is equal to $|\Sigma|+\ell$ where $\ell$ is the dimension of the linear part of the valuation cone $\calV$. Let $H\subseteq G$ be a generic stabilizer for the action of $G$ on $X$, then $\ell$ is equal to the dimension of the group $N_G(H)/H$. The torus $C$ normalizes $H$ and we have $H\cap C=\{ \Id_X\}$, hence $C$ projects to a $t$-dimensional sugroup of $N(H)/H$, whence $\ell\geq t$.
\end{proof}

\begin{definition}
For all $d=(d_1,\ldots,d_t)\in (\ZZ_{\geq0})^t$ set $\lambda(d) = d_1\lambda_1+\ldots+d_t\lambda_t$ and
\[
\calM^+_d = \calM^+ \cap \left(\lambda(d) - \cone\Sigma\right).
\]
\end{definition}
For any $\gamma\in\calM^+$ choose a $B$-eigenvector $f_\gamma\in\CC[X]$ of $B$-eigenvalue $\gamma$. Since $C\subseteq B$, the function $f_\gamma$ is multihomogeneous with respect to the sets of variables of $V_1,\ldots,V_t$.

\begin{lemma}\label{lemma:wmdecomposition}
With the above notations, the weight monoid $\calM^+$ is the disjoint union of $\calM^+_d$ for all $d\in(\ZZ_{\geq0})^t$. Moreover $\calM^+_d$ is a finite set for all $d$, and we also have
\[
\calM^+_d = \calM^+ \cap \left(\lambda(d) + \Span_\QQ\Sigma\right).
\]
\end{lemma}
\begin{proof}
First we prove that $\calM^+$ is contained in the union of $\calM^+_d$ for all $d$. For $\gamma\in\calM^+$, denote by $d=(d_1,\ldots,d_t)\in(\ZZ_{\geq0})^t$ the multidegree of $f_\gamma$, then $f_\gamma$ is in the image of the natural map
\[
\left(V_1^*\right)^{\otimes d_1}\otimes\ldots\otimes\left(V_t^*\right)^{\otimes d_t}\to \CC[X].
\]
induced by multiplication. By induction on the total degree $|d|$ of $f_\gamma$ it follows that
\[
\gamma\in \lambda(d) - \cone\Sigma,
\]
i.e.\ that $\gamma\in\calM^+_d$.

In particular $\calM_\QQ$ is generated as a rational vector space by $\lambda_1,\ldots,\lambda_t$ and $\Sigma$, and this together with Lemma~\ref{lemma:rank} yields
\begin{equation}\label{eq:directsum}
\calM_\QQ=\Span_{\QQ}\{\lambda_1,\ldots,\lambda_t\}\oplus \Span_{\QQ}\Sigma.
\end{equation}
We deduce that all sets $\calM^+_d$ are pairwise disjoint, and that the alternative formula for $\calM^+_d$ given in the statement is true. Finally, from this fact it follows also that $d$ is the multidegree of $f_\gamma$ for all $\gamma\in\calM^+_d$, hence $\calM^+_d$ only contains finitely many elements.
\end{proof}

\begin{corollary}
We have $r =t+|\Sigma|$.
\end{corollary}
\begin{proof}
This stems from (\ref{eq:directsum}) in the proof of Lemma~\ref{lemma:wmdecomposition}.
\end{proof}

We also reobtain the following result which is a special case of \cite[Lemma~7.3]{GH}.

\begin{corollary}[{\cite[see Lemma~7.3]{GH}}]\label{cor:skeletonV}
The skeleton of any spherical module $X$ is complete.
\end{corollary}
\begin{proof}
We recall that the statement means that the restrictions of the elements $\rho'(D)$ for all $D\in \calD$, generate the vector space $(\Span_\QQ\Sigma)^*$. This stems from the fact that $\calM^+_d$ is a finite set for all $d$: then for any $d$ the equations $\<\rho'(D),-\> \geq0$ (which are those defining $\calM^+$) define a bounded polytope in the affine space $\lambda(d) + \Span_\QQ\Sigma$.
\end{proof}

\begin{definition}
For any $\gamma\in\calM^+$, we define $d(\gamma)$ to be the unique element of $(\ZZ_{\geq0})^t$ such that $\gamma\in \calM^+_{d(\gamma)}$. Also, for all $d=(d_1,\ldots,d_t)\in(\ZZ_{\geq0})^t$ set $|d| = d_1+\ldots+d_t$.
\end{definition}

We extract from the proof of Lemma~\ref{lemma:wmdecomposition} that $d(\gamma)$ is just the multidegree of $f_\gamma$ with respect to the sets of variables of $V_1,\ldots,V_t$, so $|d(\gamma)|$ is the (total) degree of $f_\gamma$.

For all $d=(d_1,\ldots,d_t)\in (\ZZ_{\geq0})^t$ set
\[
\calM^+_{d,\QQ} = \left(\Span_{\QQ_{\geq 0}}\calM^+\right) \cap \left(\lambda(d) - \cone\Sigma\right).
\]
For all $\xi\in\Xi_\QQ$ define
\[
s(\xi) = \sum_{i=1}^r \<\rho(D_i),\xi\>
\]
and for all $d\in(\ZZ_{\geq0})^t$ consider the function
\[
\begin{array}{lccc}
s_d \colon & \calM^+_{d,\QQ} & \to & \QQ \\
 & \gamma & \mapsto & s(\gamma).
\end{array}
\]

\begin{lemma}\label{lemma:estimate}
Let $\gamma\in \calM^+$. If $\gamma\in \Span_{\ZZ_{\geq0}}\{\lambda_1,\ldots,\lambda_t\}$ then $|d(\gamma)|=s(\gamma)$, otherwise $|d(\gamma)|>s(\gamma)$.
\end{lemma}
\begin{proof}
Write
\[
\gamma = \sum_{i=1}^{r} c_i\lambda_i
\]
with $c_i\in\ZZ_{\geq0}$ for all $i$, then $f_\gamma$ is a scalar multiple of the function
\[
\prod_{i=1}^{r} f_{\lambda_i}^{c_i}.
\]
We have
\[
s(\gamma) = \sum_{i=1}^{r} c_i
\]
and
\[
|d(\gamma)| = \sum_{i=1}^{r} c_i\deg(f_{\lambda_i}).
\]
The lemma follows recalling that $f_{\lambda_i}$ has degree $1$ if $i\in\{1,\ldots,t\}$ and degree $\geq 2$ otherwise.
\end{proof}

\begin{corollary}\label{cor:max}
For all $d\in(\ZZ_{\geq0})^t$, the function $s_d$ attains its maximum in the unique point $\lambda(d)$, where the value of $s_d$ is equal to $|d|$.
\end{corollary}
\begin{proof}
Let $\gamma\in \calM^+_{d,\QQ}$ and choose $n\in\ZZ_{\geq1}$ such that $n\gamma\in \calM$, which implies $n\gamma\in \calM^+_{nd}$. By~Lemma~\ref{lemma:estimate} we have
\[
s_d(\gamma) = \frac1ns_{nd}(n\gamma)\leq \frac1n |nd|=|d|,
\]
and the equality is attained only in the case $n\gamma = \lambda(nd) = n\lambda(d)$ i.e.\ $\gamma=\lambda(d)$.
\end{proof}

\begin{corollary}\label{cor:1.2mfs}
If $X$ is a spherical module, the function $p_X$ attains its maximum in the unique point $\vartheta$ of the proof of Proposition~\ref{prop:mfs-geq}. In particular we have $\wp(X)=\dim(X)-\rank(X)$.
\end{corollary}
\begin{proof}
Let $\mu\in Q^*\cap \cone\Sigma$ be any point, and consider
\[
\kappa = \sum_{i=1}^r m_{D_i} \lambda_i.
\]
Notice that $\mu+\kappa$ is in $\Span_{\QQ_{\geq 0}}\calM^+$, to check this we compute
\[
\<\rho(D_i),\mu+\kappa\> = \<\rho(D_i),\mu\>+m_{D_i}\geq -m_{D_i}+m_{D_i}\geq 0
\]
for all $i$. Then $\mu+\kappa$ is in $\calM^+_{d(\kappa),\QQ}$. We have even an equality
\[
\kappa + (Q^*\cap \cone\Sigma) = \calM^+_{d(\kappa),\QQ}.
\]
Indeed, we have just proved one inclusion, and the other inclusion follows from a similar argument. Finally, we observe that
\[
p_X(\mu) = s_{d(\kappa)}(\mu+\kappa)-r
\]
for all $\mu\in Q^*\cap \cone\Sigma$. By Corollary~\ref{cor:max} the maximum of this function for $\mu\in Q^*\cap \cone\Sigma$ is attained in the unique point $\mu+\kappa = \lambda(d(\kappa))$. This point is precisely $\vartheta+\kappa$ by Lemma~\ref{lemma:wmdecomposition}, because $\vartheta+\kappa$ is in $ \Span_{\ZZ_{\geq0}}\{\lambda_1,\ldots,\lambda_t\}$ by construction, and $\vartheta$ is in $\cone\Sigma$.
\end{proof}

We end this section with some additional information on the point $\vartheta$.

\begin{lemma}\label{lemma:interior}
The point $\vartheta$ of the proof of Proposition~\ref{prop:mfs-geq} is a linear combination of spherical roots where all coefficients are strictly positive.
\end{lemma}
\begin{proof}
Recall the definition of $\kappa$ from the proof of Corollary~\ref{cor:1.2mfs}, and recall that $\vartheta+\kappa=\lambda(d(\kappa))$. Notice that $\kappa$ is in the interior of the convex cone $\Span_{\QQ_{\geq 0}}\calM^+$. Write
\[
\kappa = \underbrace{\lambda(d(\kappa))}_{\in\Span_{\ZZ_{\geq0}}\{\lambda_1,\ldots,\lambda_t\}} - \underbrace{\left(\lambda(d(\kappa))-\kappa\right)}_{=\vartheta\in\cone\Sigma}
\]
so $\kappa$ is in $\calM'_{d(\kappa)}$. There exists a rational number $\varepsilon>0$ small enough such that
\[
\kappa'=\kappa+\varepsilon \sum_{\sigma\in\Sigma}\sigma
\]
is in the cone $\Span_{\QQ_{\geq 0}}\calM^+$. So, by Lemma~\ref{lemma:wmdecomposition}, also the point $\kappa'$ is in $\calM^+_{d(\kappa),\QQ}$ which implies that
\[
\vartheta -\varepsilon \sum_{\sigma\in\Sigma}\sigma
\]
is again in $\cone\Sigma$.
\end{proof}

\section{Proof of Theorem~\ref{MT}}\label{s:proof}

From Section~\ref{s:skeleton} it follows that to prove Theorem~\ref{MT} is enough to verify Conjecture~\ref{GHC2} for the spherical skeletons coming from complete spherical varieties with a reductive generic isotropy group.

Here we start the proof by a case-by-case verification of the inequality \eqref{ineq}.

First notice that the spherical closure of a reductive spherical subgroup $H$ is reductive, since the quotient $N_GH/H$ is diagonalizable. Furthermore, if $H$ is a spherically closed reductive spherical subgroup of $G$ we can assume that $G$ is semi-simple and simply connected, with $H$ containing the whole center of $G$. Then $G\cong G_1\times\cdots\times G_k$ and $G/H$ is $G$-isomorphic to the direct product $G_1/H_1\times\cdots\times G_k/H_k$ where, for any $i=1,\ldots,k$, $H_i$ is one of the spherically closed reductive spherical subgroups occurring in the list of \cite[Section 3]{BP15}, where the corresponding (spherically closed) spherical systems are also given. 

Let us now recall some more technical definitions and simple facts from \cite[Section 8]{GH} which allow to reduce the verification of the inequality \eqref{ineq} to a finite number of cases.

A spherical skeleton $\mathcal R=(\Delta, S^p, \Sigma, \Gamma)$ is called \emph{elementary} if $\langle \rho(D), \gamma\rangle\in\{0,-1\}$ for all $D\in\Gamma$ and all $\gamma\in \Sigma$, and moreover for every $D\in\Gamma$ there exists at most one $\gamma\in\Sigma$ such that $\langle\rho(D),\gamma\rangle=-1$. An elementary spherical skeleton $\mathcal R$ is called \emph{reduced} if for every $\gamma\in\Sigma$ there exists at most one $D\in\Gamma$ such that $\langle \rho(D),\gamma\rangle=-1$.

Let $\|\Gamma\|$ denote the \emph{support} of $\Gamma$, namely, the set consisting of the elements $\gamma\in\Sigma$ such that there exists at least one $D\in\Gamma$ with $\langle \rho(D),\gamma\rangle<0$. If $\mathcal R$ is reduced elementary then $\Gamma$ is uniquely determined by its support $\|\Gamma\|$.

For every spherical skeleton $\mathcal R=(\Delta, S^p, \Sigma, \Gamma)$ one can simplify the set $\Gamma$, as done in \cite{GH}, to obtain an elementary spherical skeleton $\mathcal R^e=(\Delta, S^p, \Sigma, \Gamma^e)$ with the same underlying spherical system and $\Gamma^e=\cup_{\gamma\in\Sigma}\Gamma_\gamma$ where $\Gamma_\gamma$ is a set with $-\sum_{D\in\Gamma}\langle\rho(D),\gamma\rangle$ elements giving $-1$ on $\gamma$ and $0$ on the other spherical roots. One can further simplify the set $\Gamma^e$ to obtain a reduced elementary spherical skeleton $\mathcal R^r=(\Delta, S^p, \Sigma, \Gamma^r)$ by retaining only one element from each $\Gamma_\gamma$. The spherical skeletons $\mathcal R^r$ and $\mathcal R^e$ are complete if $\mathcal R$ is complete, and $\mathscr P(\mathcal R)\leq\mathscr P(\mathcal R^e)\leq\mathscr P(\mathcal R^r)$.

Furthermore, for two reduced elementary spherical skeletons $\mathcal R_i=(\Delta, S^p, \Sigma, \Gamma_i)$, $i=1,2$, with the same underlying spherical system, if $\|\Gamma_1\|\subset\|\Gamma_2\|$ then $\mathscr P(\mathcal R_2)\leq\mathscr P(\mathcal R_1)$. 

Therefore, it is enough to prove the inequality \eqref{ineq} for all the reduced elementary complete spherical skeletons with $\Gamma$ of minimal support and underlying spherical system corresponding to a spherically closed reductive spherical subgroup of a semi-simple simply connected complex algebraic group.

As recalled above, if $H$ is a spherically closed reductive spherical subgroup of a semi-simple simply connected $G$, the corresponding spherical system $\mathcal S=(\Delta, S^p, \Sigma)$ decomposes as the direct product of \emph{indecomposable} spherically closed spherical systems $\mathcal S_1\times\cdots\times\mathcal S_k$ occurring in the list of \cite[Section 3]{BP15}. Moreover, if $\mathcal R=(\Delta, S^p, \Sigma, \Gamma)$ is a reduced elementary spherical skeleton with $\mathcal S$ as underlying spherical system, then $\mathcal R$ necessarily decomposes into the direct product of reduced elementary spherical skeletons $\mathcal R_1\times\cdots\times\mathcal R_k$ with $\mathcal S_1,\ldots,\mathcal S_k$ underlying spherical systems, respectively, complete if $\mathcal R$ is complete, such that $\mathscr P(\mathcal R)=\mathscr P(\mathcal R_1)+\ldots+\mathscr P(\mathcal R_k)$.

It follows that we are left to verify the inequality \eqref{ineq} for all the reduced elementary complete spherical skeletons with $\Gamma$ of minimal support and underlying spherically closed spherical system coming from the list of \cite[Section 3]{BP15}.

In \cite{GH} this has already been done for the symmetric cases, thus we are left with the non-symmetric cases.

\begin{remark}
For the symmetric cases all reduced elementary spherical skeletons with $\Gamma$ of cardinality one are complete. Here in general this is no longer true, so we will have to handle with spherical skeletons with $\Gamma$ of cardinality greater than one. 
\end{remark}

\begin{remark}\label{AQ}
The combinatorics of spherical systems provides an easy way to show that certain reduced elementary spherical skeletons are not complete. Indeed, let $\Delta'\subset\Delta$ be a \emph{distinguished} subset of colors, that is, there exist $c_D>0$, for all $D\in\Delta'$, such that $\langle \rho(\sum_{D\in\Delta'}c_D D),\gamma\rangle\geq0$ for all $\gamma\in\Sigma$, with a non-trivial \emph{quotient} spherical system corresponding to a reductive spherical subgroup, that is, the set of colors $\Delta\setminus\Delta'$ and the set of spherical roots $\Sigma/\Delta'$ of the quotient (the basis of the free monoid of elements $\gamma\in\mathbb N\Sigma$ such that $\langle \rho(D),\gamma\rangle=0$ for all $D\in\Delta'$) have the following property: there exist $c_\gamma\geq0$, for all $\gamma\in\Sigma/\Delta'$, such that $\langle \rho(D),\sum_{\gamma\in\Sigma/\Delta'}c_\gamma \gamma\rangle>0$ for all $D\in\Delta\setminus\Delta'$. Let $\Sigma'$ denote the set of spherical roots $\gamma\in\Sigma$ such that $\langle \rho(\sum_{D\in\Delta'}c_D D),\gamma\rangle>0$. Then it is easy to see that every reduced elementary spherical skeleton $\mathcal R=(\Delta,S^p,\Sigma,\Gamma)$ with $\|\Gamma\|\subseteq\Sigma'$ is not complete. Since the quotients of the spherical systems corresponding to reductive spherical subgroups are known (see \cite{B13} for further details) this simplifies the computations by ruling immediately out the non-complete cases.
\end{remark}

For every reduced elementary complete spherical skeleton $\mathcal R$ with $\Gamma$ of minimal support we will compute $\mathscr P(\mathcal R)$ checking that it is less or equal than $|R^+\setminus R^+_{S^p}|$. When equality holds we will compute the (unique) point $\theta\in Q^*_{\mathcal R}\cap\mathrm{cone}(\Sigma)$ where the maximum of the functional $\sum_{D\in\mathcal D}(m_D-1+\langle\rho(D),\theta\rangle)$ is achieved, for later use.

To compute $\mathscr P(\mathcal R)$ one has to solve the linear program
\begin{equation}\label{AXB}
\max\{c^tx\ :\ x\in\mathbb R^\Sigma,\ A\,x\leq b,\ x\geq 0\}
\end{equation}
where $c$ and $x$ are column vectors in $\mathbb R^\Sigma$, $c_\gamma=\sum_{D\in\mathcal D}\langle\rho(D),\gamma\rangle$ for all $\gamma\in\Sigma$, $A$ is a matrix in $\mathbb R^{\mathcal D\times\Sigma}$, $A_{D,\gamma}=-\langle\rho(D),\gamma\rangle$ for all $D\in\mathcal D$ and $\gamma\in\Sigma$, and $b$ is a column vector in $\mathbb R^{\mathcal D}$, $b_D=m_D$ for all $D\in\mathcal D$. Then 
\[\mathscr P(\mathcal R)=\max\{c^tx\ :\ A\,x\geq b, x\geq0\}+\sum_{D\in\mathcal D}(m_D-1).\]

Equivalently, one can solve the dual linear program
\begin{equation}\label{AYC}
\min\{b^ty\ :\ y\in\mathbb R^{\mathcal D},\ A^ty\leq c,\ y\geq 0\},
\end{equation}
then
\[\mathscr P(\mathcal R)=\min\{b^ty\ :\ A^ty\leq c,\ y\geq 0\}+\sum_{D\in\mathcal D}(m_D-1).\]

\subsection*{The case-by-case computations}

We follow the list of non-symmetric spherically closed reductive spherical subgroups $H$ of semi-simple groups $G$, the cases are enumerated accordingly with \cite{BP15}: they vary from number 31 to number 39, the number 40 being not spherically closed, and from number 41 to number 50.

The isogeny type of the group $G$ is chosen just for notational convenience, it might be neither of simply connected type nor of adjoint type, but the subgroup $H$ will always contain the center $\mathrm Z_G$ of $G$. 

The simple roots are denoted $\alpha_1,\alpha_2,\ldots$ enumerated as in Bourbaki and $\alpha'_1,\alpha'_2,\ldots$, $\alpha''_1,\ldots$ in further irreducible components.

The spherical system $(\Delta, S^p, \Sigma)$ will be described explicitly; if it exists, we will also give a distinguished subset of colors $\Delta'\subset\Delta$ with a quotient corresponding to a reductive spherical subgroup together with the corresponding subset of spherical roots $\Sigma'\subset\Sigma$, see Remark~\ref{AQ} above.

For the acquainted reader we include the Luna diagrams, whose significance will not be recalled here, see \emph{loc.cit.}

\subsection*{31}

$\GL(1)\times\Sp(2p)\subset\SL(2p+1)$, $p\geq2$:
$\Sigma=\{\gamma_i=\alpha_i+\alpha_{i+1}\ :\ 1\leq i\leq 2p-1\}$, $\Delta=\{D_i\ :\ 1\leq i\leq 2p\}$, $\rho(D_i)=\alpha^\vee_i$ $\forall\ i$, $S^p=\emptyset$, $\Delta'=\Delta\setminus\{D_1,D_{2p}\}$, $\Sigma'=\{\gamma_{2k}\ :\ 1\leq k\leq p-1\}$.

\[\begin{picture}(27000,1800)\put(0,600){\diagramacastn}\put(12000,900){\vector(1,0){3000}}
\put(15900,900){
\multiput(0,0)(1800,0){2}{\usebox{\edge}}\put(3600,0){\usebox{\susp}}\put(7200,0){\usebox{\dynkinathree}}\put(10800,0){\usebox{\wcircle}}\put(0,0){\usebox{\wcircle}}
\put(0,0){\multiput(0,0)(25,25){13}{\circle*{70}}\multiput(300,300)(300,0){34}{\multiput(0,0)(25,-25){7}{\circle*{70}}}\multiput(450,150)(300,0){34}{\multiput(0,0)(25,25){7}{\circle*{70}}}\multiput(10800,0)(-25,25){13}{\circle*{70}}}
}
\end{picture}\] 

Let us solve the linear program \eqref{AYC} for the reduced elementary spherical skeletons with $\Gamma$ of cardinality one and support in $\Sigma\setminus\Sigma'$. Set $y_i=y_{D_i}$ for $1\leq i\leq 2p$ and $y_{2p+1}=y_D$ where $D$ is the unique element of $\Gamma$.   

If $\|\Gamma\|=\{\gamma_{2p-1}\}$ (the case $\|\Gamma\|=\{\gamma_{1}\}$ is equivalent by symmetry), we have for all $1\leq\ell<p$ (by induction on $\ell$)
\begin{eqnarray*}
&&y_{2\ell+1}\geq y_{2\ell-1}+y_{2\ell}-y_{2\ell-2}\geq\ell y_2+y_1+\ell\geq\ell\\
&&y_{2\ell+2}\geq y_{2\ell}+y_{2\ell+1}-y_{2\ell-1}\geq(\ell +1)y_2+\ell\geq\ell
\end{eqnarray*}
and
\[y_{2p+1}\geq y_{2p-1}+y_{2p}-y_{2p-2}\geq p\, y_2 + y_1 +p\geq p.\]
Then $\mathscr P=2p^2+p=|R^+|$. Dually, $\mathscr P$ is the maximum of $2p+x_{\gamma_1}$ which is achieved in 
\[\theta=\sum_{i=1}^{2p-1}\frac{i(i+1)}2\gamma_{2p-i}.\]

If $\|\Gamma\|=\{\gamma_{2k-1}\}$ (can assume $p/2 < k < p$), we have for all $1\leq\ell<k$  
\begin{eqnarray*}
&&y_{2\ell+1}\geq y_{2\ell-1}+y_{2\ell}-y_{2\ell-2}\geq\ell y_2+y_1+\ell\geq\ell\\
&&y_{2\ell+2}\geq y_{2\ell}+y_{2\ell+1}-y_{2\ell-1}\geq(\ell +1)y_2+\ell\geq\ell
\end{eqnarray*}
and similarly for $1 \leq m\leq p-k$, set $y'_i=y_{2p+1-i}$ for all $i$, we have
\begin{eqnarray*}
&&y'_{2m+2}\geq y'_{2m}+y'_{2m+1}-y'_{2m-1}\geq(m+1) y'_2+y_1+m\geq m\\
&&y'_{2m+1}\geq y'_{2m-1}+y'_{2m}-y'_{2m-2}\geq m\,y'_2+y'_1+m\geq m
\end{eqnarray*}
and
\[y_{2p+1}+y_{2k+1}\geq y_{2k-2}+y_{2k-1}-y_{2k-3}-1\geq y_2 + y_{2k-1}\geq k-1.\]
Then $\mathscr P=2p^2+p-2(p-k)(2k+1)<|R^+|$, for all $p/2< k < p$, and $\mathscr P=2p^2+p-2(k-1)(2p-2k+3)<|R^+|$, for all $1< k \leq p/2$.

\subsection*{32}

$\GL(p)\subset\SO(2p+1)$, $p\geq2$: 
$\Sigma = \{ \gamma_i = \alpha_i+\alpha_{i+1} \ :\ 1\leq i\leq p-1\}\cup \{ \gamma_p = \alpha_p\}$, 
$\Delta=\{D_i\ :\ 1\leq i\leq p-1\}\cup\{D_p^+,D_p^-\}$, $\rho(D_i)=\alpha^\vee_i$ $\forall\ i\leq p-1$, $\rho(D_p^+)=\rho(D_p^-)$ and $\rho(D_p^+)+\rho(D_p^-)=\alpha_p^\vee$, $S^p=\emptyset$, $\Delta'=\Delta\setminus\{D_1,D_p^-\}$, $\Sigma'=\{\gamma_{2k}\ :\ 1\leq k\leq p/2\}$.

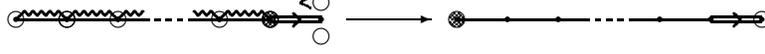
\[\begin{picture}(27000,1800)
\put(300,900){
\put(0,0){\usebox{\atwo}}
\put(1800,0){\usebox{\atwoseq}}
\put(9000,0){\usebox{\btwo}}
\put(10800,0){\usebox{\aone}}
\put(10800,600){\usebox{\tow}}
}
\put(12000,900){\vector(1,0){3000}}
\put(15900,900){
\multiput(0,0)(1800,0){2}{\usebox{\edge}}\put(3600,0){\usebox{\susp}}\put(7200,0){\usebox{\dynkinbthree}}\put(0,0){\usebox{\gcircle}}\put(10800,0){\usebox{\wcircle}}
}
\end{picture}\]

Here $|R^+|=p^2$. Set $x_i=x_{\gamma_i}$ for $1\leq i\leq p$.   

Suppose $\|\Gamma\|=\{\gamma_1\}$: we are computing the sup of the function $(p-1)+(x_p-x_{p-2})$. The first two inequalities of the system are:
\[
x_1\leq 1,\; x_2\leq x_1+2
\]
Others are:
\[
x_k-x_{k-2}\leq x_{k-1}-x_{k-3}+2, \;\; \forall k\in \{ 4, \ldots, p\}
\]
which imply $x_p-x_{p-2}\leq 1+2(p-1)=2p-1$. It is also straightforward to check that this upper bound is attained, which yields $\mathscr P=3p-2$, and we have $\mathscr P \leq |R^+|$, with equality only for $p=2$. In the latter case the maximum is attained in 
\[\theta=\gamma_1+3\gamma_2.\]

Suppose that $\|\Gamma\|=\{ \gamma_{\ell} \}$ with $\ell$ odd and $\ell\geq 3$. We are computing the sup of the function $(p-1)+(x_p-x_{p-2}+x_1-x_\ell)$; notice that it could be $p=\ell$ or $p-2=\ell$, which will not influence our estimates. Write $z_k = x_k - x_{k-2}$, so the function is
\[
(p-1)+(z_p - (z_\ell+z_{\ell -2}+\ldots+z_3)).
\]
The next-to-last inequality of the system is
\[
(-z_p)\leq 1
\]
and some of the others are
\[
(-z_k)\leq (-z_{k+1})+2
\]
for all $k<p$, which imply
\[
(-z_k) \leq (-z_p)+2(p-k)\leq 1+2(p-k).
\]
Combining these and setting $\ell=2m+1$ we get
\[
(p-1)+(z_p - (z_\ell+z_{\ell -2}+\ldots+z_3)) \leq
\]
\[
z_p+1+2(p-\ell)+1+2(p-\ell-2)+\ldots+1+2(p-5)-z_p+2(p-3)+(p-1)=
\]
\[
m-1+2pm-2m(m+2)+(p-1) = p(2m+1)+m(-2m-3)-1< p\ell\leq p^2.
\]
which proves the claim with a strict inequality.

\subsection*{33}

$\N(\GL(p))\subset\SO(2p+1)$, $p\geq2$:
$\Sigma = \{ \gamma_i = \alpha_i+\alpha_{i+1} \ :\ 1\leq i\leq p-1\}\cup \{ \gamma_p = 2\alpha_p\}$, 
$\Delta=\{D_i\ :\ 1\leq i\leq p\}$, $\rho(D_i)=\alpha^\vee_i$ $\forall\ i\leq p-1$  and $\rho(D_p)=\frac12\alpha_p^\vee$, $S^p=\emptyset$, $\Delta'=\Delta\setminus\{D_1\}$, $\Sigma'=\{\gamma_{2k}\ :\ 1\leq k\leq p/2\}$.

\[\begin{picture}(31800,5700)
\put(0,1950){\diagrambcprimen}
\put(12000,3450){\vector(3,1){3000}}
\put(12000,2250){\vector(3,-1){3000}}
\put(15900,4800){
\multiput(0,0)(1800,0){2}{\usebox{\edge}}\put(3600,0){\usebox{\susp}}\put(7200,0){\usebox{\dynkinbthree}}\put(0,0){\usebox{\gcircle}}\put(11400,-150){if $p$ is even}
}
\put(15900,900){
\multiput(0,0)(1800,0){2}{\usebox{\edge}}\put(3600,0){\usebox{\susp}}\put(7200,0){\usebox{\dynkinbthree}}\put(0,0){\usebox{\gcircletwo}}\put(11400,-150){if $p$ is odd}
}
\end{picture}\]

Here $|R^+|=p^2$ as in the previous case. Set $x_i=x_{\gamma_i}$ for $1\leq i\leq p$. 

With any $p$, if $\|\Gamma\|=\{\gamma_1\}$ then we are computing the sup of the constant function $p-1$ hence $\mathscr P=p-1<p^2$. If instead $\|\Gamma\|=\{\gamma_\ell\}$ with $\ell$ odd and $3\leq \ell \leq p$, we are computing the sup of the function $(p-1)+(x_1-x_\ell)$.

Set $z_k=x_k-x_{k+2}$, then the function is
\[
(p-1)+(z_1+z_3+\ldots+z_{\ell-2})
\]

Some of the inequalities of the system are
\[
z_1 \leq z_{2}+2, \; \ldots, \; z_{p-4}\leq z_{p-3}+2
\]
and
\[
z_{p-3}\leq z_{p-2}-x_p+2\leq 3.
\]
whence
\[
z_i \leq z_{p-3}+2(p-3-i)
\]
if $i\leq p-4$. Suppose $\ell\leq p-1$ and set $\ell=2m+1$, we apply the above inequalities to our function obtaining
\[
(p-1)+(z_1+z_3+\ldots+z_{\ell-2}) \leq (p-1)+3m+\sum_{j=1}^m(2(p-3-2j+1))=
\]
\[
=p\ell - m - 2m(m-1)-1<p\ell<p^2.
\]
If instead we have $\ell=p$, then the function is
\[
(p-1)+(z_1+z_3+\ldots+z_{p-4})+z_{p-2}<\ldots
\]
We may apply the above estimate to the sum of the two terms in parentheses, using $\ell'=p-2$ instead of $\ell$. This, together with $x_\ell=x_p\leq 1$ which is the last inequality of the system, yields
\[
\ldots < p(p-2)+z_{p-2}\leq p(p-2)+2<p^2
\]
proving the claim with a strict inequality.

\subsection*{34}

 $\mathrm{Spin}(7)\subset\SO(9)$:
$\Sigma=\{\gamma_1=\alpha_1+\alpha_2+\alpha_3+\alpha_4,\ \gamma_2=\alpha_2+2\alpha_3+3\alpha_4\}$, $\Delta=\{D_1,D_4\}$, $\rho(D_i)=\alpha^\vee_i$ $\forall\ i\in\{1,4\}$, $S^p=\{\alpha_2,\alpha_3\}$, $\Delta'=\{D_4\}$, $\Sigma'=\{\gamma_2\}$.

\[\begin{picture}(16800,1800)(0,-900)\put(300,0){\put(0,0){\usebox{\dynkinbfour}}\multiput(0,0)(5400,0){2}{\usebox{\gcircle}}}\put(6900,0){\vector(1,0){3000}}\put(11100,0){\put(0,0){\usebox{\dynkinbfour}}\put(0,0){\usebox{\gcircle}}}\end{picture}\]

If $\|\Gamma\|=\{\gamma_1\}$ then $\mathscr P=13=|R^+\setminus R_{S^p}^+|$. The maximum of $8+x_{\gamma_2}$ is achieved in 
\[\theta=\gamma_1+5\gamma_2.\]

\subsection*{35}

$\mathsf G_2\subset \SO(7)$:
$\Sigma=\{\gamma = \alpha_1+2\alpha_2+3\alpha_3\}$, $\Delta=\{D_3\}$, $\rho(D_3)=\alpha^\vee_3$, $S^p=\{\alpha_1,\alpha_2\}$.

\[\begin{picture}(3900,1800)\put(0,900){\usebox{\bthirdthree}}\end{picture}\]

For $\|\Gamma\|=\{\gamma\}$ we have $\mathscr P=6=|R^+\setminus R_{S^p}^+|$, maximum of $5+x_{\gamma}$ achieved in 
\[\theta=\gamma.\]

\subsection*{36}

$\GL(1)\times\Sp(2p)\subset\Sp(2p+2)$, $p\geq2$:
$\Sigma=\{\gamma_1=\alpha_1,\ \gamma_2=\alpha_1+2\alpha_2+\ldots+2\alpha_p+\alpha_{p+1}\}$, $\Delta=\{D_1^+,D_1^-,D_2\}$, $\rho(D_1^+)=\rho(D_1^-)$, $\rho(D_1^+)+\rho(D_1^-)=\alpha_1^\vee$ and $\rho(D_2)=\alpha_2^\vee$, $S^p=\{\alpha_3,\ldots,\alpha_{p+1}\}$, $\Delta'=\{D_1^+,D_1^-\}$, $\Sigma'=\{\gamma_1\}$.

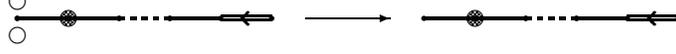
\[\begin{picture}(24000,1800)(0,-900)\put(300,0){\put(0,0){\usebox{\aone}}\put(0,0){\usebox{\shortcm}}}\put(10500,0){\vector(1,0){3000}}\put(14700,0){\put(0,0){\usebox{\shortcm}}}\end{picture}\]

If $\|\Gamma\|=\{\gamma_2\}$ then $\mathscr P=4p=|R^+\setminus R_{S^p}^+|$. The maximum of $2p-1+x_{\gamma_1}$ is achieved in 
\[\theta=(2p+1)\gamma_1+\gamma_2.\]

\subsection*{37}

$\N(\GL(1)\times\Sp(2p))\subset\Sp(2p+2)$, $p\geq2$:
$\Sigma=\{\gamma_1=2\alpha_1,\ \gamma_2=\alpha_1+2\alpha_2+\ldots+2\alpha_p+\alpha_{p+1}\}$, $\Delta=\{D_1,D_2\}$, $\rho(D_1)=\frac12\alpha_1^\vee$ and $\rho(D_2)=\alpha_2^\vee$, $S^p=\{\alpha_3,\ldots,\alpha_{p+1}\}$, $\Delta'=\{D_1\}$, $\Sigma'=\{\gamma_1\}$.

\[\begin{picture}(24000,1800)(0,-900)\put(300,0){\put(0,0){\usebox{\aprime}}\put(0,0){\usebox{\shortcm}}}\put(10500,0){\vector(1,0){3000}}\put(14700,0){\put(0,0){\usebox{\shortcm}}}\end{picture}\]

If $\|\Gamma\|=\{\gamma_2\}$ then $\mathscr P=2p<|R^+\setminus R_{S^p}^+|$.

\subsection*{38}

$\mathsf G_2\cdot \mathrm Z_{\SO(8)}\subset\SO(8)$:
$\Sigma=\{\gamma_1=\alpha_1+\alpha_2+\alpha_3,\ \gamma_2=\alpha_1+\alpha_2+\alpha_4,\ \gamma_3=\alpha_2+\alpha_3+\alpha_4\}$, $\Delta=\{D_1,D_3,D_4\}$, $\rho(D_i)=\alpha_i^\vee$ $\forall i\in\{1,3,4\}$, $S^p=\{\alpha_2\}$. As distinguished subset of colors one can take $\Delta'=\{D_3,D_4\}$ getting  $\Sigma'=\{\gamma_3\}$ or, by symmetry, any subset $\Delta'$ of cardinality 2, hence getting any subset $\Sigma'$ of cardinality 1.

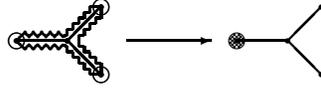
\[\begin{picture}(11400,3000)
\put(0,0){\diagramdsastfour}
\put(4200,1500){\vector(1,0){3000}}
\put(8100,1500){
\put(0,0){\usebox{\dynkindfour}}\put(0,0){\usebox{\gcircle}}
}
\end{picture}\]

For $|\Gamma|=2$ we have $\mathscr P=11= |R^+\setminus R_{S^p}^+|$. If $\|\Gamma\|=\{\gamma_1,\gamma_2\}$, the maximun of $6+x_{\gamma_3}$ is achieved in 
\[\theta=\gamma_1+\gamma_2+5\gamma_3.\]

\subsection*{39}

$\SO(2)\times\mathrm{Spin}(7)\subset\SO(10)$:
$\Sigma=\{\gamma_1=\alpha_1,\ \gamma_2=\alpha_2+\alpha_3+\alpha_4,\ \gamma_3=\alpha_2+\alpha_3+\alpha_5,\ \gamma_4=\alpha_3+\alpha_4+\alpha_5\}$, $\Delta=\{D_1^+,D_1^-,D_2,D_4,D_5\}$, $\rho(D_1^+)=1,-1,0,0$ and $\rho(D_1^-)=1,0,-1,0$ respectively on $\gamma_1,\gamma_2,\gamma_3,\gamma_4$, $\rho(D_i)=\alpha_i^\vee$ $\forall i\in\{2,4,5\}$, $S^p=\{\alpha_3\}$, $\Delta'=\{D_4,D_5\}$, $\Sigma'=\{\gamma_4\}$.

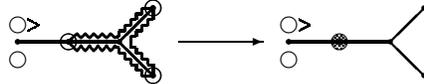
\[\begin{picture}(15000,3000)
\put(300,1500){
\put(0,0){\usebox{\edge}}
\put(0,0){\usebox{\aone}}
\put(0,600){\usebox{\toe}}
\put(1500,-1500){\diagramdsastfour}
}
\put(6000,1500){\vector(1,0){3000}}
\put(9900,1500){
\put(0,0){\usebox{\dynkindfive}}\put(0,0){\usebox{\aone}}\put(0,600){\usebox{\toe}}\put(1800,0){\usebox{\gcircle}}
}
\end{picture}\]

Here $|R^+\setminus R_{S^p}^+|=19$. 
If $\|\Gamma\|=\{\gamma_1\}$ then $\mathscr P=12$.
If $\|\Gamma\|=\{\gamma_2\}$ or $\|\Gamma\|=\{\gamma_3\}$ then $\mathscr P=18$.

\subsection*{41} 

$\mathsf A_2\subset\mathsf G_2$:
$\Sigma=\{\gamma = 4\alpha_1+2\alpha_2\}$, $\Delta=\{D_1\}$, $\rho(D_1)=\alpha^\vee_1$, $S^p=\{\alpha_2\}$.

\[\begin{picture}(2100,1800)(-300,-900)\put(0,0){\usebox{\gprimetwo}}\end{picture}\]

For $\|\Gamma\|=\{\gamma\}$ we have $\mathscr P=5=|R^+\setminus R_{S^p}^+|$, maximum of $4+x_{\gamma}$ achieved in 
\[\theta=\gamma.\]

\subsection*{42} 

$\Sp(2)\times\Sp(2p)\times\Sp(2q)\subset\Sp(2p+2)\times\Sp(2q+2)$, $p\geq0,q\geq1$, 
where $\Sp(2)$ is diagonal in $\Sp(2p+2)\times\Sp(2q+2)$.

\paragraph{$\mathbf{(p=0)}$:}
$\Sigma=\{\gamma_1=\alpha_1+\alpha'_1,\ \gamma_2=\alpha'_1+2\alpha'_2+\ldots+2\alpha'_q+\alpha'_{q+1}\}$, $\Delta=\{D'_1,D'_2\}$, $\rho(D'_i)=(\alpha'_i)^\vee$ $\forall i\in\{1,2\}$, $S^p=\{\alpha'_3,\ldots,\alpha'_{q+1}\}$, $\Delta'=\{D'_1\}$, $\Sigma'=\{\gamma_1\}$.

\[\begin{picture}(29400,1800)(-300,-900)
\put(0,0){\usebox{\vertex}}
\put(0,0){\usebox{\wcircle}}
\put(2700,0){\usebox{\shortcsecondm}}
\multiput(0,-300)(2700,0){2}{\line(0,-1){600}}
\put(0,-900){\line(1,0){2700}}
\put(13200,0){\vector(1,0){3000}}
\put(17400,0){
\put(0,0){\usebox{\vertex}}
\put(2700,0){\usebox{\shortcm}}
}
\end{picture}\]

If $\|\Gamma\|=\{\gamma_2\}$ then $\mathscr P=4q+1=|R^+\setminus R_{S^p}^+|$. The maximun of $2q+x_{\gamma_1}$ is achieved in 
\[\theta=(2q+1)\gamma_1+\gamma_2.\]

\paragraph{$\mathbf{(p\geq1)}$:}
$\Sigma=\{\gamma_1=\alpha_1+\alpha'_1,\ \gamma_2=\alpha_1+2\alpha_2+\ldots+2\alpha_p+\alpha_{p+1},\ \gamma_3=\alpha'_1+2\alpha'_2+\ldots+2\alpha'_q+\alpha'_{q+1}\}$, $\Delta=\{D_1,D_2,D'_2\}$, $\rho(D_i)=\alpha_i^\vee$ $\forall i\in\{1,2\}$ and $\rho(D'_2)=(\alpha'_2)^\vee$, $S^p=\{\alpha_3,\ldots,\alpha_{p+1},\ \alpha'_3,\ldots,\alpha'_{q+1}\}$. As distinguished subset of colors one can take $\Delta'=\{D_1,D_2\}$ getting  $\Sigma'=\{\gamma_1,\gamma_2\}$ or, by symmetry, $\Delta'=\{D_1,D'_2\}$ getting $\Sigma'=\{\gamma_1,\gamma_3\}$.

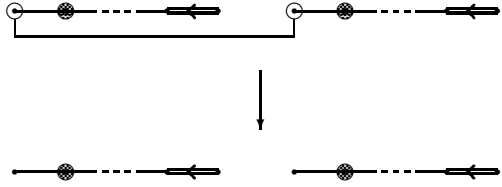
\[\begin{picture}(17400,7500)(-300,-6600)
\multiput(0,0)(9900,0){2}{
\put(0,0){\usebox{\edge}}
\put(1800,0){\usebox{\susp}}
\put(5400,0){\usebox{\leftbiedge}}
\put(1800,0){\usebox{\gcircle}}
}
\multiput(0,0)(9900,0){2}{\usebox{\wcircle}}
\multiput(0,-300)(9900,0){2}{\line(0,-1){600}}
\put(0,-900){\line(1,0){9900}}
\put(8700,-2100){\vector(0,-1){2100}}
\put(0,-5700){
\multiput(0,0)(9900,0){2}{
\put(0,0){\usebox{\edge}}
\put(1800,0){\usebox{\susp}}
\put(5400,0){\usebox{\leftbiedge}}
\put(1800,0){\usebox{\gcircle}}
}
}
\end{picture}\]

If $\|\Gamma\|=\{\gamma_2,\gamma_3\}$ then $\mathscr P=2(p+q-1)<|R^+\setminus R_{S^p}^+|$.

\subsection*{43}  

$(\Sp(2)\times\Sp(2p)\times\Sp(2q)\times\Sp(2r))\cdot\mathrm Z\subset\Sp(2p+2)\times\Sp(2q+2)\times\Sp(2r+2)$, 
$p,q,r\geq 0$, where $\Sp(2)$ is diagonal in $\Sp(2p+2)\times\Sp(2q+2)\times\Sp(2r+2)$.

\paragraph{$\mathbf{(p=q=r=0)}$:}
$\Sigma=\{\alpha_1,\alpha'_1,\alpha''_1\}$, $\Delta=\{D,D',D''\}$, $\rho(D)=1,1,-1$, $\rho(D')=1,-1,1$ and $\rho(D'')=-1,1,1$ respectively on $\alpha_1,\alpha'_1,\alpha''_1$, $S^p=\emptyset$. As distinguished subset of colors one can take $\Delta'=\{D',D''\}$ getting  $\Sigma'=\{\alpha''_1\}$ or, by symmetry, any subset $\Delta'$ of cardinality 2, hence getting any subset $\Sigma'$ of cardinality 1.
  
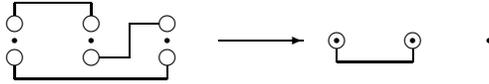
\[\begin{picture}(17400,2700)(-300,-1350)
\multiput(0,0)(2700,0){3}{\usebox{\aone}}
\multiput(0,1350)(2700,0){2}{\line(0,-1){450}}
\put(0,1350){\line(1,0){2700}}
\multiput(0,-1350)(5400,0){2}{\line(0,1){450}}
\put(0,-1350){\line(1,0){5400}}
\multiput(3000,-600)(1050,1200){2}{\line(1,0){1050}}
\put(4050,-600){\line(0,1){1200}}
\put(7200,0){\vector(1,0){3000}}
\put(11400,0){
\multiput(0,0)(2700,0){3}{\usebox{\vertex}}
\multiput(0,0)(2700,0){2}{\usebox{\wcircle}}
\multiput(0,-750)(2700,0){2}{\line(0,1){450}}
\put(0,-750){\line(1,0){2700}}
}
\end{picture}\]

For $|\Gamma|=2$ we have $\mathscr P=3= |R^+|$. If $\|\Gamma\|=\{\alpha_1,\alpha'_1\}$, the maximun of $x_{\alpha''_1}$ is achieved in 
\[\theta=\alpha_1+\alpha'_1+3\alpha''_1.\]

\paragraph{$\mathbf{(p\neq0,\ q=r=0)}$:}
$\Sigma=\{\gamma_1=\alpha_1,\ \gamma_2=\alpha'_1,\ \gamma_3=\alpha''_1,\ \gamma_4=\alpha''_1+2\alpha''_2+\ldots+2\alpha''_p+\alpha''_{p+1}\}$, $\Delta=\{D,D',D'',D''_2\}$, $\rho(D)=1,1,-1,0$, $\rho(D')=1,-1,1,0$, $\rho(D'')=-1,1,1,0$ respectively on $\gamma_1,\gamma_2,\gamma_3,\gamma_4$ and $\rho(D''_2)=(\alpha''_2)^\vee$, $S^p=\{\alpha''_3,\ldots,\alpha''_{p+1}\}$. As distinguished subset of colors one can take $\Delta'=\{D',D'',D''_2\}$ or $\Delta'=\{D,D',D''\}$ getting, respectively, $\Sigma'=\{\gamma_3,\gamma_4\}$ and $\Sigma'=\{\gamma_1,\gamma_2,\gamma_3\}$.
  
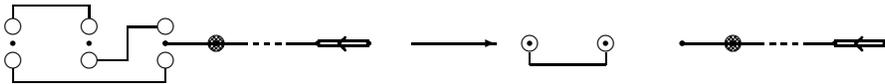
\[\begin{picture}(31200,2700)(-300,-1350)
\multiput(0,0)(2700,0){3}{\usebox{\aone}}
\multiput(0,1350)(2700,0){2}{\line(0,-1){450}}
\put(0,1350){\line(1,0){2700}}
\multiput(0,-1350)(5400,0){2}{\line(0,1){450}}
\put(0,-1350){\line(1,0){5400}}
\multiput(3000,-600)(1050,1200){2}{\line(1,0){1050}}
\put(4050,-600){\line(0,1){1200}}
\put(5400,0){\usebox{\shortshortcm}}
\put(14100,0){\vector(1,0){3000}}
\put(18300,0){
\multiput(0,0)(2700,0){3}{\usebox{\vertex}}
\multiput(0,0)(2700,0){2}{\usebox{\wcircle}}
\multiput(0,-750)(2700,0){2}{\line(0,1){450}}
\put(0,-750){\line(1,0){2700}}
\put(5400,0){\usebox{\shortshortcm}}
}
\end{picture}\]

If $\|\Gamma\|$ is equal to $\{\gamma_1,\ \gamma_4\}$ or equal to $\{\gamma_2,\ \gamma_4\}$ then $\mathscr P=4p+2=|R^+\setminus R_{S^p}^+|$. If $\|\Gamma\|=\{\gamma_1,\ \gamma_4\}$, the maximum of $2p-1+x_{\gamma_2}$ is achieved in
\[\theta=\gamma_1+(2p+3)\gamma_2+(2p+1)\gamma_3+\gamma_4.\]

\paragraph{$\mathbf{(p,q\neq0,\ r=0)}$:}
$\Sigma=\{\gamma_1=\alpha_1,\ \gamma_2=\alpha'_1,\ \gamma_3=\alpha'_1+2\alpha'_2+\ldots+2\alpha'_p+\alpha'_{p+1},\ \gamma_4=\alpha''_1,\ \gamma_5=\alpha''_1+2\alpha''_2+\ldots+2\alpha''_q+\alpha''_{q+1}\}$, $\Delta=\{D,D',D'_2,D'',D''_2\}$, $\rho(D)=1,1,0,-1,0$, $\rho(D')=1,-1,0,1,0$, $\rho(D'')=-1,1,0,1,0$ respectively on $\gamma_1,\ldots,\gamma_5$, $\rho(D'_2)=(\alpha'_2)^\vee$ and $\rho(D''_2)=(\alpha''_2)^\vee$, $S^p=\{\alpha'_3,\ldots,\alpha'_{p+1},\;\alpha''_3,\ldots,\alpha''_{q+1}\}$. As distinguished subset of colors one can take $\Delta'=\Delta\setminus\{D'_2\}$ or $\Delta'=\Delta\setminus\{D''_2\}$ getting, respectively, $\Sigma'=\Sigma\setminus\{\gamma_3\}$ and $\Sigma'=\Sigma\setminus\{\gamma_5\}$.
  
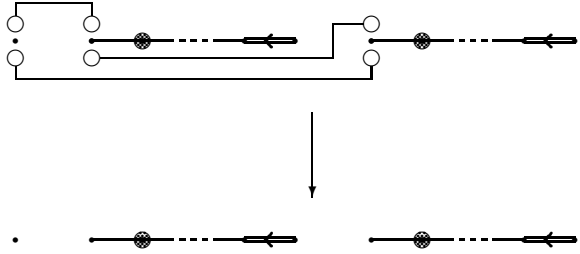
\[\begin{picture}(20100,9300)(-300,-7950)
\multiput(0,0)(2700,0){2}{\usebox{\aone}}
\put(12600,0){\usebox{\aone}}
\multiput(0,1350)(2700,0){2}{\line(0,-1){450}}
\put(0,1350){\line(1,0){2700}}
\multiput(0,-1350)(12600,0){2}{\line(0,1){450}}
\put(0,-1350){\line(1,0){12600}}
\put(3000,-600){\line(1,0){8250}}
\put(11250,600){\line(1,0){1050}}
\put(11250,-600){\line(0,1){1200}}
\multiput(2700,0)(9900,0){2}{
\put(0,0){\usebox{\edge}}
\put(1800,0){\usebox{\susp}}
\put(5400,0){\usebox{\leftbiedge}}
\put(1800,0){\usebox{\gcircle}}
}
\put(10500,-2550){\vector(0,-1){3000}}
\put(0,-7050){
\put(0,0){\usebox{\vertex}}
\multiput(2700,0)(9900,0){2}{
\put(0,0){\usebox{\edge}}
\put(1800,0){\usebox{\susp}}
\put(5400,0){\usebox{\leftbiedge}}
\put(1800,0){\usebox{\gcircle}}
}
}
\end{picture}\]

If $\|\Gamma\|=\{\gamma_3,\gamma_5\}$ then $\mathscr P=4p+4q+1=|R^+\setminus R_{S^p}^+|$. The maximum of $2p+2q-2-1+x_{\gamma_1}$ is achieved in
\[\theta=(2p+2q+3)\gamma_1+(2p+1)\gamma_2+\gamma_3+(2q+1)\gamma_4+\gamma_5.\]

\paragraph{$\mathbf{(p,q,r\neq0)}$:}
$\Sigma=\{\gamma_1=\alpha_1,\ \gamma_2=\alpha_1+2\alpha_2+\ldots+2\alpha_p+\alpha_{p+1},\ \gamma_3=\alpha'_1,\ \gamma_4=\alpha'_1+2\alpha'_2+\ldots+2\alpha'_q+\alpha'_{q+1},\ \gamma_5=\alpha''_1,\ \gamma_6=\alpha''_1+2\alpha''_2+\ldots+2\alpha''_r+\alpha''_{r+1}\}$, $\Delta=\{D,D_2,D',D'_2,D'',D''_2\}$, $\rho(D)=1,0,1,0,-1,0$, $\rho(D')=1,0,-1,0,1,0$, $\rho(D'')=-1,0,1,0,1,0$ respectively on $\gamma_1,\ldots,\gamma_6$, $\rho(D_2)=(\alpha_2)^\vee$, $\rho(D'_2)=(\alpha'_2)^\vee$ and $\rho(D''_2)=(\alpha''_2)^\vee$, $S^p=\{\alpha_3,\ldots,\alpha_{p+1},\;\alpha'_3,\ldots,\alpha'_{q+1},\;\alpha''_3,\ldots,\alpha''_{r+1}\}$. As distinguished subset of colors one can take either $\Delta'=\Delta\setminus\{D_2\}$ or $\Delta'=\Delta\setminus\{D'_2\}$ or $\Delta'=\Delta\setminus\{D''_2\}$ getting, respectively, $\Sigma'=\Sigma\setminus\{\gamma_2\}$, $\Sigma'=\Sigma\setminus\{\gamma_4\}$ and $\Sigma'=\Sigma\setminus\{\gamma_6\}$.

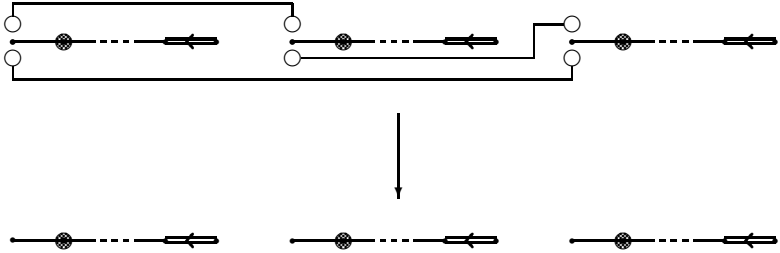
\[\begin{picture}(27300,9300)(-300,-7950)
\multiput(0,0)(9900,0){3}{\usebox{\aone}}
\multiput(0,1350)(9900,0){2}{\line(0,-1){450}}
\put(0,1350){\line(1,0){9900}}
\multiput(0,-1350)(19800,0){2}{\line(0,1){450}}
\put(0,-1350){\line(1,0){19800}}
\put(10200,-600){\line(1,0){8250}}
\put(18450,600){\line(1,0){1050}}
\put(18450,-600){\line(0,1){1200}}
\multiput(0,0)(9900,0){3}{
\put(0,0){\usebox{\edge}}
\put(1800,0){\usebox{\susp}}
\put(5400,0){\usebox{\leftbiedge}}
\put(1800,0){\usebox{\gcircle}}
}
\put(13650,-2550){\vector(0,-1){3000}}
\put(0,-7050){
\multiput(0,0)(9900,0){3}{
\put(0,0){\usebox{\edge}}
\put(1800,0){\usebox{\susp}}
\put(5400,0){\usebox{\leftbiedge}}
\put(1800,0){\usebox{\gcircle}}
}
}\end{picture}\]

If $\|\Gamma\|=\{\gamma_2,\gamma_4,\gamma_6\}$ then $\mathscr P=2(p+q+r)-3<|R^+\setminus R_{S^p}^+|$.
 
\bigskip

\subsection*{44} 

$\GL(1)\times\SL(2)\times\SL(p)\subset\SL(p+2)\times\SL(2)$, $p\geq2$, where $\SL(2)$ is diagonal in $\SL(p+2)\times\SL(2)$.

\paragraph{$\mathbf{(p=2)}$:}
$\Sigma=\{\alpha_1, \alpha_2, \alpha_3, \alpha'_1\}$, $\Delta=\{D^+,D_2^+,D_2^-,D',D''\}$, $\rho(D^+)=1,-1,1,-1$, $\rho(D_2^+)=-1,1,0,0$, $\rho(D_2^-)=0,1,-1,0$, $\rho(D')=1,0,-1,1$ and $\rho(D'')=-1,0,1,1$ respectively on $\alpha_1,\alpha_2,\alpha_3,\alpha'_1$, $S^p=\emptyset$, $\Delta'=\{D',D''\}$, $\Sigma'=\{\alpha'_1\}$.
 
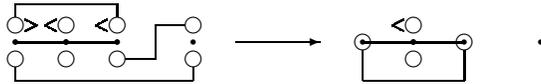
\[\begin{picture}(19200,2700)(-300,-1350)
\multiput(0,0)(3600,0){2}{\usebox{\aone}}
\multiput(0,0)(1800,0){2}{\usebox{\edge}}
\put(1800,0){\usebox{\aone}}
\put(6300,0){\usebox{\aone}}
\multiput(0,1350)(3600,0){2}{\line(0,-1){450}}
\put(0,1350){\line(1,0){3600}}
\multiput(0,-1350)(6300,0){2}{\line(0,1){450}}
\put(0,-1350){\line(1,0){6300}}
\multiput(3900,-600)(1050,1200){2}{\line(1,0){1050}}
\put(4950,-600){\line(0,1){1200}}
\put(0,600){\usebox{\toe}}
\put(3600,600){\usebox{\tow}}
\put(1800,600){\usebox{\tow}}
\put(7800,0){\vector(1,0){3000}}
\put(12300,0){
\multiput(0,0)(3600,0){2}{\usebox{\wcircle}}
\multiput(0,0)(1800,0){2}{\usebox{\edge}}
\put(1800,0){\usebox{\aone}}
\put(6300,0){\usebox{\vertex}}
\multiput(0,-300)(3600,0){2}{\line(0,-1){1050}}
\put(0,-1350){\line(1,0){3600}}
\put(1800,600){\usebox{\tow}}
}
\end{picture}\]

Here $|R^+\setminus R_{S^p}^+|=7$. 
If $\|\Gamma\|=\{\alpha_1\}$ or $\|\Gamma\|=\{\alpha_3\}$ then $\mathscr P=6$.
If $\|\Gamma\|=\{\alpha_2\}$  then $\mathscr P=4$.

\paragraph{$\mathbf{(p\geq3)}$:}
$\Sigma=\{\gamma_1=\alpha_1,\ \gamma_2=\alpha_2+\ldots+\alpha_p,\ \gamma_3=\alpha_{p+1},\ \gamma_4=\alpha'_1\}$, $\Delta=\{D^+,D_2,D_p,D',D''\}$, $\rho(D^+)=1,-1,1,-1$, $\rho(D_2)=-1,1,0,0$, $\rho(D_p)=0,1,-1,0$, $\rho(D')=1,0,-1,1$ and $\rho(D'')=-1,0,1,1$ respectively on $\gamma_1,\gamma_2,\gamma_3,\gamma_4$ (as in the previous case), $S^p=\{\alpha_3,\ldots,\alpha_{p-1}\}$, $\Delta'=\{D',D''\}$, $\Sigma'=\{\gamma_4\}$.
 
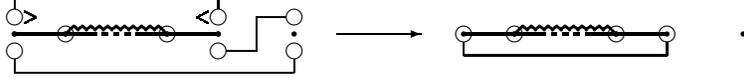
\[\begin{picture}(26400,2700)(-300,-1350)
\multiput(0,0)(7200,0){2}{\usebox{\aone}}
\multiput(0,0)(5400,0){2}{\usebox{\edge}}
\put(1800,0){\usebox{\shortam}}
\put(9900,0){\usebox{\aone}}
\multiput(0,1350)(7200,0){2}{\line(0,-1){450}}
\put(0,1350){\line(1,0){7200}}
\multiput(0,-1350)(9900,0){2}{\line(0,1){450}}
\put(0,-1350){\line(1,0){9900}}
\multiput(7500,-600)(1050,1200){2}{\line(1,0){1050}}
\put(8550,-600){\line(0,1){1200}}
\put(0,600){\usebox{\toe}}
\put(7200,600){\usebox{\tow}}
\put(11400,0){\vector(1,0){3000}}
\put(15900,0){
\multiput(0,0)(7200,0){2}{\usebox{\wcircle}}
\multiput(0,0)(5400,0){2}{\usebox{\edge}}
\put(1800,0){\usebox{\shortam}}
\put(9900,0){\usebox{\vertex}}
\multiput(0,-300)(7200,0){2}{\line(0,-1){450}}
\put(0,-750){\line(1,0){7200}}
}
\end{picture}\]

Here $|R^+\setminus R_{S^p}^+|=4p-1$. 
If $\|\Gamma\|=\{\gamma_1\}$ or $\|\Gamma\|=\{\gamma_3\}$ then $\mathscr P=3p$.
If $\|\Gamma\|=\{\gamma_2\}$  then $\mathscr P=4(p-1)$.

\subsection*{45} 

$\GL(1)\times\SL(2)\times\SL(p)\times\Sp(2q)\subset\SL(p+2)\times\Sp(2q+2)$, $p\geq1$, $q\geq1$, where $\SL(2)$ is diagonal in $\SL(p+2)\times\Sp(2q+2)$.

\paragraph{$\mathbf{(p=1)}$:}
$\Sigma=\{\gamma_1=\alpha_1,\ \gamma_2=\alpha_2,\ \gamma_3=\alpha'_1,\ \gamma_4=\alpha'_1+2\alpha'_2+\ldots+2\alpha'_q+\alpha'_{q+1}\}$, $\Delta=\{D_1^+, D_1^-, D_2^+, D_2^-, D'\}$, $\rho(D_1^+)=1,-1,1,0$, $\rho(D_1^-)=1,0,-1,0$, $\rho(D_2^+)=0,1,-1,0$ and $\rho(D_2^-)=-1,1,1,0$ respectively on $\gamma_1,\gamma_2,\gamma_3,\gamma_4$, $\rho(D')=(\alpha'_2)^\vee$, $S^p=\{\alpha'_3,\ldots,\alpha'_{q+1}\}$. As distinguished subset of colors one can take $\Delta'=\Delta\setminus\{D_1^-,D_2^+\}$ or $\Delta'=\Delta\setminus\{D'\}$ obtaining, respectively, $\Sigma'=\{\gamma_3,\gamma_4\}$ and $\Sigma'=\{\gamma_1,\gamma_2,\gamma_3\}$.

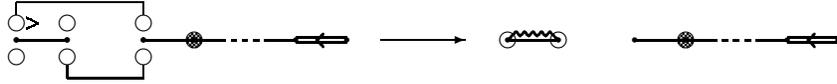
\[\begin{picture}(29400,2700)(-300,-1350)
\put(0,0){\usebox{\edge}}
\multiput(0,0)(1800,0){2}{\usebox{\aone}}
\put(4500,0){\usebox{\aone}}
\put(4500,0){\usebox{\shortshortcm}}
\multiput(0,1350)(4500,0){2}{\line(0,-1){450}}
\put(0,1350){\line(1,0){4500}}
\multiput(1800,-1350)(2700,0){2}{\line(0,1){450}}
\put(1800,-1350){\line(1,0){2700}}
\put(0,600){\usebox{\toe}}
\put(12900,0){\vector(1,0){3000}}
\put(17400,0){
\put(0,0){\usebox{\atwo}}
\put(4500,0){\usebox{\shortshortcm}}
}
\end{picture}\]

If $\|\Gamma\|$ is equal to $\{\gamma_1,\gamma_4\}$ or equal to $\{\gamma_2,\gamma_4\}$ then $\mathscr P=2q+1<|R^+\setminus R_{S^p}^+|$.

\paragraph{$\mathbf{(p=2)}$:}
$\Sigma=\{\gamma_1=\alpha_1,\ \gamma_2=\alpha_2,\ \gamma_3=alpha_3,\ \gamma_4=\alpha'_1,\ \gamma_5=\alpha'_1+2\alpha'_2+\ldots+2\alpha'_q+\alpha'_{q+1}\}$, $\Delta=\{D_1^+, D_1^-, D_2^+, D_2^-,D_3^-, D'\}$, $\rho(D_1^+)=1,-1,1,-1,0$, $\rho(D_1^-)=1,0,-1,1,0$, $\rho(D_2^+)=-1,1,0,0,0$, $\rho(D_2^-)=0,1,-1,0,0$ and $\rho(D_3^-)=-1,0,1,1,0$ respectively on $\gamma_1,\ldots,\gamma_5$, $\rho(D')=(\alpha'_2)^\vee$, $S^p=\{\alpha'_3,\ldots,\alpha'_{q+1}\}$. As distinguished subset of colors one can take $\Delta'=\{D_1^-,D_3^-,D'\}$ or $\Delta'=\Delta\setminus\{D'\}$ obtaining, respectively, $\Sigma'=\{\gamma_4,\gamma_5\}$ and $\Sigma'=\Sigma\setminus\{\gamma_5\}$.

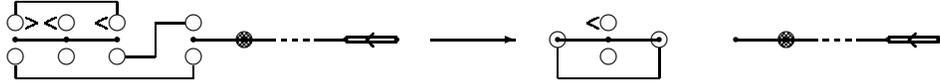
\[\begin{picture}(33000,2700)(-300,-1350)
\multiput(0,0)(3600,0){2}{\usebox{\aone}}
\multiput(0,0)(1800,0){2}{\usebox{\edge}}
\put(1800,0){\usebox{\aone}}
\put(6300,0){\usebox{\aone}}
\multiput(0,1350)(3600,0){2}{\line(0,-1){450}}
\put(0,1350){\line(1,0){3600}}
\multiput(0,-1350)(6300,0){2}{\line(0,1){450}}
\put(0,-1350){\line(1,0){6300}}
\multiput(3900,-600)(1050,1200){2}{\line(1,0){1050}}
\put(4950,-600){\line(0,1){1200}}
\put(0,600){\usebox{\toe}}
\put(3600,600){\usebox{\tow}}
\put(1800,600){\usebox{\tow}}
\put(6300,0){\usebox{\shortshortcm}}
\put(14700,0){\vector(1,0){3000}}
\put(19200,0){
\multiput(0,0)(3600,0){2}{\usebox{\wcircle}}
\multiput(0,0)(1800,0){2}{\usebox{\edge}}
\put(1800,0){\usebox{\aone}}
\multiput(0,-300)(3600,0){2}{\line(0,-1){1050}}
\put(0,-1350){\line(1,0){3600}}
\put(1800,600){\usebox{\tow}}
\put(6300,0){\usebox{\shortshortcm}}
}
\end{picture}\]

Here $|R^+\setminus R_{S^p}^+|=6+4q$. If $\|\Gamma\|$ is equal to $\{\gamma_1,\gamma_5\}$ or equal to $\{\gamma_3,\gamma_5\}$ then $\mathscr P=2q+2$.
If $\|\Gamma\|=\{\gamma_2,\gamma_5\}$ then $\mathscr P=2q-1$.

\paragraph{$\mathbf{(p\geq3)}$:}
$\Sigma=\{\gamma_1=\alpha_1,\ \gamma_2=\alpha_2+\ldots+\alpha_p,\ \gamma_3=alpha_{p+1},\ \gamma_4=\alpha'_1,\ \gamma_5=\alpha'_1+2\alpha'_2+\ldots+2\alpha'_q+\alpha'_{q+1}\}$, $\Delta=\{D_1^+, D_1^-, D_2, D_p,D_{p+1}^-, D'\}$, $\rho(D_1^+)=1,-1,1,-1,0$, $\rho(D_1^-)=1,0,-1,1,0$, $\rho(D_2)=-1,1,0,0,0$, $\rho(D_p)=0,1,-1,0,0$ and $\rho(D_{p+1}^-)=-1,0,1,1,0$ respectively on $\gamma_1,\ldots,\gamma_5$ (as in the previous case), $\rho(D')=(\alpha'_2)^\vee$, $S^p=\{\alpha_3,\ldots,\alpha_{p-1},\;\alpha'_3,\ldots,\alpha'_{q+1}\}$. As distinguished subset of colors one can take $\Delta'=\{D_1^-,D_{p+1}^-,D'\}$ or $\Delta'=\Delta\setminus\{D'\}$ obtaining, respectively, $\Sigma'=\{\gamma_4,\gamma_5\}$ and $\Sigma'=\Sigma\setminus\{\gamma_5\}$.

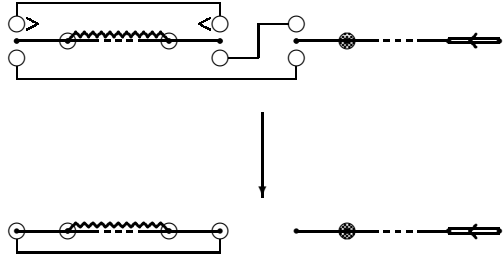
\[\begin{picture}(17400,9600)(-300,-8250)
\multiput(0,0)(7200,0){2}{\usebox{\aone}}
\multiput(0,0)(5400,0){2}{\usebox{\edge}}
\put(1800,0){\usebox{\shortam}}
\put(9900,0){\usebox{\aone}}
\multiput(0,1350)(7200,0){2}{\line(0,-1){450}}
\put(0,1350){\line(1,0){7200}}
\multiput(0,-1350)(9900,0){2}{\line(0,1){450}}
\put(0,-1350){\line(1,0){9900}}
\multiput(7500,-600)(1050,1200){2}{\line(1,0){1050}}
\put(8550,-600){\line(0,1){1200}}
\put(0,600){\usebox{\toe}}
\put(7200,600){\usebox{\tow}}
\put(9900,0){\usebox{\shortshortcm}}
\put(8700,-2550){\vector(0,-1){3000}}
\put(0,-6750){
\multiput(0,0)(7200,0){2}{\usebox{\wcircle}}
\multiput(0,0)(5400,0){2}{\usebox{\edge}}
\put(1800,0){\usebox{\shortam}}
\multiput(0,-300)(7200,0){2}{\line(0,-1){450}}
\put(0,-750){\line(1,0){7200}}
\put(9900,0){\usebox{\shortshortcm}}
}
\end{picture}\]

Here $|R^+\setminus R_{S^p}^+|=4p+4q-2$. If $\|\Gamma\|$ is equal to $\{\gamma_1,\gamma_5\}$ or equal to $\{\gamma_3,\gamma_5\}$ then $\mathscr P=2(p+q-1)$.
If $\|\Gamma\|=\{\gamma_2,\gamma_5\}$ then $\mathscr P=2p+2q-5$.

\subsection*{46} 

$\N(\SO(p))\subset\SO(p+1)\times\SO(p)$, diagonally, $4\leq p\leq 6$.

\paragraph{$\mathbf{(p=4)}$:}
$\Sigma=\{\alpha_1,\alpha_2,\alpha'_1,\alpha''_1\}$, $\Delta=\{D_1^+,D_1^-,D_2^+,D_2^-\}$, $\rho(D_1^+)=1,-1,1,1$, $\rho(D_1^-)=1,0,-1,-1$, $\rho(D_2^+)=-1,1,1,-1$ and $\rho(D_2^-)=-1,1,-1,1$ respectively on $\alpha_1$, $\alpha_2$, $\alpha'_1$, $\alpha''_1$, $S^p=\emptyset$, $\Delta'=\Delta\setminus\{D_1^-\}$, $\Sigma'=\{\alpha'_1,\alpha''_1\}$.

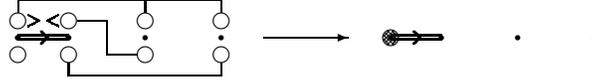
\[\begin{picture}(21000,2700)(-300,-1350)
\put(0,0){\usebox{\dynkinbtwo}}
\multiput(0,0)(1800,0){2}{\usebox{\aone}}
\multiput(4500,0)(2700,0){2}{\usebox{\aone}}
\put(0,1350){\line(0,-1){450}}
\multiput(4500,1350)(2700,0){2}{\line(0,-1){450}}
\put(0,1350){\line(1,0){7200}}
\multiput(2100,600)(1050,-1200){2}{\line(1,0){1050}}
\put(3150,600){\line(0,-1){1200}}
\multiput(1800,-1350)(5400,0){2}{\line(0,1){450}}
\put(1800,-1350){\line(1,0){5400}}
\put(0,600){\usebox{\toe}}
\put(1800,600){\usebox{\tow}}
\put(8700,0){\vector(1,0){3000}}
\put(13200,0){
\put(0,0){\usebox{\dynkinbtwo}}
\put(0,0){\usebox{\gcircle}}
\multiput(4500,0)(2700,0){2}{\usebox{\vertex}}
}
\end{picture}\]

Here $|R^+\setminus R_{S^p}^+|=6$. If $\|\Gamma\|=\{\alpha_1\}$ then $\mathscr P=3$.
If $\|\Gamma\|=\{\alpha_2\}$ then $\mathscr P=0$.

\paragraph{$\mathbf{(p=5)}$:}
$\Sigma=\{\alpha_1,\alpha_2,\alpha'_1,\alpha'_2,\alpha'_3\}$, $\Delta=\{D_1^+,D_1^-,D_2^+,D_2^-,(D'_2)^+\}$, $\rho(D_1^+)=1,-1,1,-1,1$, $\rho(D_1^-)=1,0,-1,1,-1$, $\rho(D_2^+)=-1,1,1,0,-1$, $\rho(D_2^-)=-1,1,-1,0,1$ and $\rho((D'_2)^+)=-1,0,0,1,0$ respectively on $\alpha_1,\ldots,\alpha'_3$, $S^p=\emptyset$, $\Delta'=\Delta\setminus\{(D'_2)^+\}$, $\Sigma'=\{\alpha_1,\alpha_2\}$.

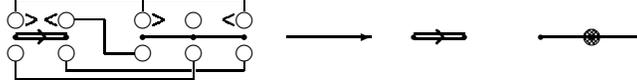
\[\begin{picture}(22800,2850)(-300,-1500)
\put(0,0){\usebox{\dynkinbtwo}}
\put(4500,0){\usebox{\dynkinathree}}
\multiput(0,0)(1800,0){2}{\usebox{\aone}}
\multiput(4500,0)(1800,0){3}{\usebox{\aone}}
\put(0,1350){\line(0,-1){450}}
\multiput(4500,1350)(3600,0){2}{\line(0,-1){450}}
\put(0,1350){\line(1,0){8100}}
\multiput(2100,600)(1050,-1200){2}{\line(1,0){1050}}
\put(3150,600){\line(0,-1){1200}}
\multiput(0,-1500)(6300,0){2}{\line(0,1){600}}
\put(0,-1500){\line(1,0){6300}}
\multiput(1800,-1200)(6300,0){2}{\line(0,1){300}}
\put(1800,-1200){\line(1,0){4400}}
\put(6400,-1200){\line(1,0){1700}}
\multiput(0,600)(4500,0){2}{\usebox{\toe}}
\multiput(1800,600)(6300,0){2}{\usebox{\tow}}
\put(9600,0){\vector(1,0){3000}}
\put(14100,0){
\put(0,0){\usebox{\dynkinbtwo}}
\put(4500,0){\usebox{\dynkinathree}}
\put(6300,0){\usebox{\gcircle}}
}
\end{picture}\]

If $\|\Gamma\|=\{\alpha'_1\}$ or $\|\Gamma\|=\{\alpha'_3\}$ then $\mathscr P=10=|R^+|$.
If $\|\Gamma\|=\{\alpha'_2\}$ then $\mathscr P=4$. If $\|\Gamma\|=\{\alpha'_1\}$, the maximum of $-x_{\alpha_1}+x_{\alpha_2}-x_{\alpha'_1}+x_{\alpha'_2}$ is achieved in 
\[\theta=5\alpha_1+12\alpha_2+\alpha'_1+4\alpha'_2+9\alpha'_3.\]
 
\paragraph{$\mathbf{(p=6)}$:}
$\Sigma=\{\alpha_1,\alpha_2,\alpha_3,\alpha'_1,\alpha'_2,\alpha'_3\}$, $\Delta=\{D_1^+,D_1^-,D_2^+,D_3^+,D_3^-\}$, $\rho(D_1^+)=1,-1,0,0,1,0$, $\rho(D_1^-)=1,0,0,0,-1,0$, $\rho(D_2^+)=0,1,-1,1,-1,1$, $\rho(D_2^-)=-1,1,0,-1,1,-1$, $\rho(D_3^+)=0,-1,1,1,0,-1$ and $\rho(D_3^-)=0,-1,1,-1,0,1$ respectively on $\alpha_1,\ldots,\alpha'_3$, $S^p=\emptyset$, $\Delta'=\Delta\setminus\{D_1^-\}$, $\Sigma'=\{\alpha'_1,\alpha'_2,\alpha'_3\}$.

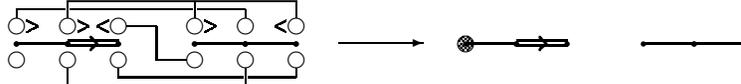
\[\begin{picture}(26400,3000)(-2100,-1500)
\put(-1800,0){\usebox{\dynkinbthree}}
\put(4500,0){\usebox{\dynkinathree}}
\multiput(-1800,0)(1800,0){3}{\usebox{\aone}}
\multiput(4500,0)(1800,0){3}{\usebox{\aone}}
\put(0,1500){\line(0,-1){600}}
\multiput(4500,1500)(3600,0){2}{\line(0,-1){600}}
\put(0,1500){\line(1,0){8100}}
\multiput(2100,600)(1050,-1200){2}{\line(1,0){1050}}
\put(3150,600){\line(0,-1){1200}}
\multiput(-1800,1200)(8100,0){2}{\line(0,-1){300}}
\multiput(-1800,1200)(6400,0){2}{\line(1,0){1700}}
\put(100,1200){\line(1,0){4300}}
\multiput(0,-1500)(6300,0){2}{\line(0,1){600}}
\put(0,-1500){\line(1,0){6300}}
\multiput(1800,-1200)(6300,0){2}{\line(0,1){300}}
\put(1800,-1200){\line(1,0){4400}}
\put(6400,-1200){\line(1,0){1700}}
\put(-1800,600){\usebox{\toe}}
\multiput(0,600)(4500,0){2}{\usebox{\toe}}
\multiput(1800,600)(6300,0){2}{\usebox{\tow}}
\put(9600,0){\vector(1,0){3000}}
\put(15900,0){
\put(-1800,0){\usebox{\dynkinbthree}}
\put(4500,0){\usebox{\dynkinathree}}
\put(-1800,0){\usebox{\gcircle}}
}
\end{picture}\]

Here $|R^+\setminus R_{S^p}^+|=15$. If $\|\Gamma\|=\{\alpha_1\}$ then $\mathscr P=5$.
If $\|\Gamma\|=\{\alpha_2\}$ then $\mathscr P=2$.
If $\|\Gamma\|=\{\alpha_3\}$ then $\mathscr P=3$.

\subsection*{47} 

$(\Sp(2)\times\Sp(2)\times\Sp(2p)\times\Sp(2q))\cdot\mathrm Z\subset\Sp(4)\times\Sp(2p+2)\times\Sp(2q+2)$, 
$p\geq0,q\geq1$, where $\Sp(2)\times\Sp(2)$ is embedded in $\Sp(4)$, 
then the first factor $\Sp(2)$ is also embedded into $\Sp(2p+2)$ (complementary to $\Sp(2p)$), 
and the second into $\Sp(2q+2)$ (complementary to $\Sp(2q)$).

\paragraph{$\mathbf{(p=0)}$:}
$\Sigma=\{\gamma_1=\alpha_1,\ \gamma_2=\alpha_2,\ \gamma_3=\alpha'_1,\ \gamma_4=\alpha''_1,\ \gamma_5=\alpha''_1+2\alpha''_2+\ldots+2\alpha''_q+\alpha''_{q+1}\}$, $\Delta=\{D_1^+,D_1^-,D_2^+,D_2^-,D''\}$, $\rho(D_1^+)=1,-1,1,1,0$, $\rho(D_1^-)=1,0,-1,-1,0$, $\rho(D_2^+)=-1,1,1,-1,0$, $\rho(D_2^-)=-1,1,-1,1,0$ respectively on $\gamma_1,\ldots,\gamma_5$ and $\rho(D'')=(\alpha''_2)^\vee$, $S^p=\{\alpha''_3,\ldots,\alpha''_{q+1}\}$. As distinguished subset of colors one can take $\Delta'=\Delta\setminus\{D_1^-\}$ or $\Delta'=\Delta\setminus\{D''\}$ getting, respectively, $\Sigma'=\{\gamma_3,\gamma_4,\gamma_5\}$ and $\Sigma'=\Sigma\setminus\{\gamma_5\}$.

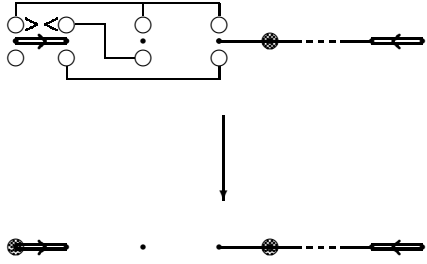
\[\begin{picture}(14700,10200)(-300,-8850)
\put(0,0){\usebox{\dynkinbtwo}}
\multiput(0,0)(1800,0){2}{\usebox{\aone}}
\multiput(4500,0)(2700,0){2}{\usebox{\aone}}
\put(0,1350){\line(0,-1){450}}
\multiput(4500,1350)(2700,0){2}{\line(0,-1){450}}
\put(0,1350){\line(1,0){7200}}
\multiput(2100,600)(1050,-1200){2}{\line(1,0){1050}}
\put(3150,600){\line(0,-1){1200}}
\multiput(1800,-1350)(5400,0){2}{\line(0,1){450}}
\put(1800,-1350){\line(1,0){5400}}
\put(0,600){\usebox{\toe}}
\put(1800,600){\usebox{\tow}}
\put(7200,0){\usebox{\shortshortcm}}
\put(7350,-2650){\vector(0,-1){3000}}
\put(0,-7300){
\put(0,0){\usebox{\dynkinbtwo}}
\put(0,0){\usebox{\gcircle}}
\put(4500,0){\usebox{\vertex}}
\put(7200,0){\usebox{\shortshortcm}}
}
\end{picture}\]

Here $|R^+\setminus R_{S^p}^+|=4q+5$. If $\|\Gamma\|=\{\gamma_1,\gamma_5\}$ then $\mathscr P=2(q+1)$.
If $\|\Gamma\|=\{\gamma_2,\gamma_5\}$ then $\mathscr P=2q-1$.

\paragraph{$\mathbf{(p\geq1)}$:}
$\Sigma=\{\gamma_1=\alpha_1,\ \gamma_2=\alpha_2,\ \gamma_3=\alpha'_1,\ \gamma_4=\alpha'_1+2\alpha'_2+\ldots+2\alpha'_p+\alpha'_{p+1},\ \gamma_5=\alpha''_1,\ \gamma_6=\alpha''_1+2\alpha''_2+\ldots+2\alpha''_q+\alpha''_{q+1}\}$, $\Delta=\{D_1^+,D_1^-,D_2^+,D_2^-,D',D''\}$, $\rho(D_1^+)=1,-1,1,0,1,0$, $\rho(D_1^-)=1,0,-1,0,-1,0$, $\rho(D_2^+)=-1,1,1,0,-1,0$, $\rho(D_2^-)=-1,1,-1,0,1,0$ respectively on $\gamma_1,\ldots,\gamma_6$, $\rho(D')=(\alpha'_2)^\vee$ and $\rho(D'')=(\alpha''_2)^\vee$, $S^p=\{\alpha'_3,\ldots,\alpha'_{p+1},\;\alpha''_3,\ldots,\alpha''_{q+1}\}$. As distinguished subset of colors one can take either $\Delta'=\Delta\setminus\{D_1^-\}$ or $\Delta'=\Delta\setminus\{D'\}$ or $\Delta'=\Delta\setminus\{D''\}$ getting, respectively, $\Sigma'=\Sigma\setminus\{\gamma_1,\gamma_2\}$, $\Sigma'=\Sigma\setminus\{\gamma_4\}$ and $\Sigma'=\Sigma\setminus\{\gamma_6\}$.

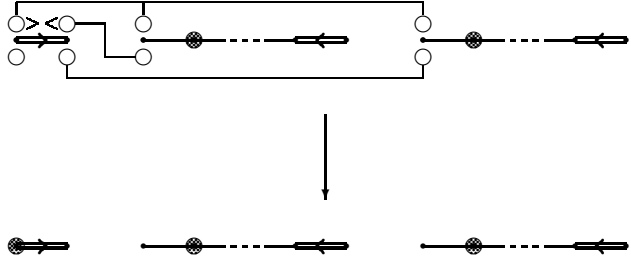
\[\begin{picture}(21900,10200)(-300,-8850)
\put(0,0){\usebox{\dynkinbtwo}}
\multiput(0,0)(1800,0){2}{\usebox{\aone}}
\multiput(4500,0)(9900,0){2}{\usebox{\aone}}
\put(0,1350){\line(0,-1){450}}
\multiput(4500,1350)(9900,0){2}{\line(0,-1){450}}
\put(0,1350){\line(1,0){14400}}
\multiput(2100,600)(1050,-1200){2}{\line(1,0){1050}}
\put(3150,600){\line(0,-1){1200}}
\multiput(1800,-1350)(12600,0){2}{\line(0,1){450}}
\put(1800,-1350){\line(1,0){12600}}
\put(0,600){\usebox{\toe}}
\put(1800,600){\usebox{\tow}}
\multiput(4500,0)(9900,0){2}{
\put(0,0){\usebox{\edge}}
\put(1800,0){\usebox{\susp}}
\put(5400,0){\usebox{\leftbiedge}}
\put(1800,0){\usebox{\gcircle}}
}
\put(10950,-2650){\vector(0,-1){3000}}
\put(0,-7300){
\put(0,0){\usebox{\dynkinbtwo}}
\put(0,0){\usebox{\gcircle}}
\multiput(4500,0)(9900,0){2}{
\put(0,0){\usebox{\edge}}
\put(1800,0){\usebox{\susp}}
\put(5400,0){\usebox{\leftbiedge}}
\put(1800,0){\usebox{\gcircle}}
}
}
\end{picture}\]

Here $|R^+\setminus R_{S^p}^+|=4(p+q+1)$. If $\|\Gamma\|=\{\gamma_1,\gamma_4,\gamma_6\}$ then $\mathscr P=2p+2q-1$.
If $\|\Gamma\|=\{\gamma_2,\gamma_4,\gamma_6\}$ then $\mathscr P=2(p+q-1)$.

\subsection*{48} 

$(\Sp(4)\times\Sp(2p))\cdot \mathrm Z\subset\Sp(4)\times\Sp(2p+4)$, $p\geq1$, where $\Sp(4)$ is diagonal in $\Sp(4)\times\Sp(2p+4)$.

\paragraph{$\mathbf{(p=1)}$:}
$\Sigma=\{\alpha_1,\alpha_2,\alpha'_1,\alpha'_2,\alpha'_3\}$, $\Delta=\{D_1^+,D_1^-,D_2^+,D_2^-,(D'_2)^+\}$, $\rho(D_1^+)=1,-1,1,-1,1$, $\rho(D_1^-)=1,0,-1,1,-1$, $\rho(D_2^+)=-1,1,1,0,-1$, $\rho(D_2^-)=-1,1,-1,0,1$ and $\rho((D'_2)^+)=-1,0,0,1,-1$ respectively on $\alpha_1,\ldots,\alpha'_3$, $S^p=\emptyset$, $\Delta'=\Delta\setminus\{(D'_2)^+\}$, $\Sigma'=\{\alpha_1,\alpha_2\}$.

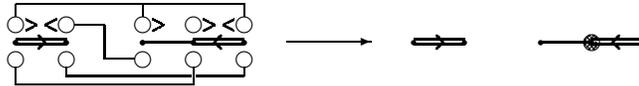
\[\begin{picture}(22800,2850)(-300,-1500)
\put(0,0){\usebox{\dynkinbtwo}}
\put(4500,0){\usebox{\dynkincthree}}
\multiput(0,0)(1800,0){2}{\usebox{\aone}}
\multiput(4500,0)(1800,0){3}{\usebox{\aone}}
\put(0,1350){\line(0,-1){450}}
\multiput(4500,1350)(3600,0){2}{\line(0,-1){450}}
\put(0,1350){\line(1,0){8100}}
\multiput(2100,600)(1050,-1200){2}{\line(1,0){1050}}
\put(3150,600){\line(0,-1){1200}}
\multiput(0,-1500)(6300,0){2}{\line(0,1){600}}
\put(0,-1500){\line(1,0){6300}}
\multiput(1800,-1200)(6300,0){2}{\line(0,1){300}}
\put(1800,-1200){\line(1,0){4400}}
\put(6400,-1200){\line(1,0){1700}}
\multiput(0,600)(4500,0){2}{\usebox{\toe}}
\multiput(1800,600)(6300,0){2}{\usebox{\tow}}
\put(6300,600){\usebox{\toe}}
\put(9600,0){\vector(1,0){3000}}
\put(14100,0){
\put(0,0){\usebox{\dynkinbtwo}}
\put(4500,0){\usebox{\dynkincthree}}
\put(6300,0){\usebox{\gcircle}}
}
\end{picture}\]

Here $|R^+\setminus R^+_{S^p}|=13$. If $\|\Gamma\|$ is equal to $\{\alpha'_1\}$, $\{\alpha'_2\}$ or $\{\alpha'_3\}$ then $\mathscr P$ equals, respectively, $1$, $4$, $8$.
 
\paragraph{$\mathbf{(p\geq1)}$:}
$\Sigma=\{\gamma_1=\alpha_1,\ \gamma_2=\alpha_2,\ \gamma_3=\alpha'_1,\ \gamma_4=\alpha'_2,\ \gamma_5=\alpha'_3,\ \gamma_6=\alpha'_3+2\alpha'_4+\ldots+2\alpha'_{p+1}+\alpha'_{p+2}\}$, $\Delta=\{D_1^+,D_1^-,D_2^+,D_2^-,(D'_2)^+,D'_4\}$, $\rho(D_1^+)=1,-1,1,-1,1,0$, $\rho(D_1^-)=1,0,-1,1,-1,0$, $\rho(D_2^+)=-1,1,1,0,-1,0$, $\rho(D_2^-)=-1,1,-1,0,1,0$ and $\rho((D'_2)^+)=-1,0,0,1,0,-1$ respectively on $\gamma_1,\ldots,\gamma_6$, $\rho(D'_4)=(\alpha'_4)^\vee$, $S^p=\{\alpha'_5,\ldots,\alpha'_{p+1}\}$, $\Delta'=\Delta\setminus\{(D'_2)^+,D'_4\}$, $\Sigma'=\{\gamma_1,\gamma_2\}$.

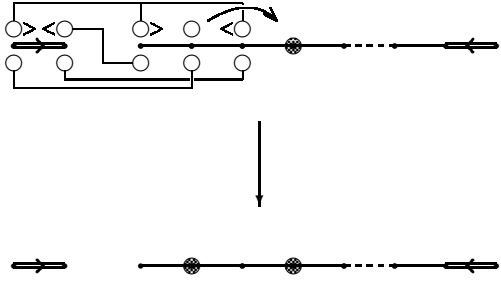
\[\begin{picture}(17400,10200)(-300,-9000)
\put(0,0){\usebox{\dynkinbtwo}}
\put(4500,0){\usebox{\dynkinathree}}
\multiput(0,0)(1800,0){2}{\usebox{\aone}}
\multiput(4500,0)(1800,0){3}{\usebox{\aone}}
\multiput(0,1500)(4500,0){2}{\line(0,-1){600}}
\put(8100,1500){\line(0,-1){50}}
\put(8100,1250){\line(0,-1){350}}
\put(0,1500){\line(1,0){8100}}
\multiput(2100,600)(1050,-1200){2}{\line(1,0){1050}}
\put(3150,600){\line(0,-1){1200}}
\multiput(0,-1500)(6300,0){2}{\line(0,1){600}}
\put(0,-1500){\line(1,0){6300}}
\multiput(1800,-1200)(6300,0){2}{\line(0,1){300}}
\put(1800,-1200){\line(1,0){4400}}
\put(6400,-1200){\line(1,0){1700}}
\multiput(0,600)(4500,0){2}{\usebox{\toe}}
\multiput(1800,600)(6300,0){2}{\usebox{\tow}}
\put(6300,600){\usebox{\tobe}}
\put(8100,0){\usebox{\shortcm}}
\put(8700,-2700){\vector(0,-1){3000}}
\put(0,-7800){
\put(0,0){\usebox{\dynkinbtwo}}
\put(4500,0){\usebox{\dynkinathree}}
\put(6300,0){\usebox{\gcircle}}
\put(8100,0){\usebox{\shortcm}}
}
\end{picture}\]

Here $|R^+\setminus R^+_{S^p}|=8p+4$. If $\|\Gamma\|$ is equal to $\{\gamma_3\}$, $\{\gamma_4\}$, $\{\gamma_5\}$ or $\{\gamma_6\}$ then $\mathscr P$ equals, respectively, $2p-2$, $2p+1$, $2p+6$, $8p-1$.
 
\subsection*{49} 

$\SL(p)\times\GL(1)\subset \SL(p)\times\SL(p+1)$, $p\geq2$, where $\SL(p)$ is diagonal in $\SL(p)\times\SL(p+1)$:
$\Sigma=\{\alpha_1,\ldots,\alpha_{p-1},\;\alpha'_1,\ldots,\alpha'_p\}$, $\Delta=\{(D'_i)^+\ :\ 1\leq i\leq p\}\cup\{(D'_i)^-\ :\ 1\leq i\leq p\}$, $\rho((D'_i)^+)=-1,1,-1,1$ respectively on $\alpha_{p-i},\alpha_{p-i+1},\alpha'_{i-1},\alpha'_i$ and $\rho((D'_i)^-)=1,-1,1,-1$ respectively on $\alpha_{p-i},\alpha_{p-i+1},\alpha'_{i},\alpha'_{i+1}$ for all $1\leq i\leq p$, and $0$ on the other spherical roots, $S^p=\emptyset$, $\Delta'=\Delta\setminus\{(D'_1)^+,(D'_{p})^-\}$, $\Sigma'=\{\alpha_1,\ldots,\alpha_{p-1}\}$.

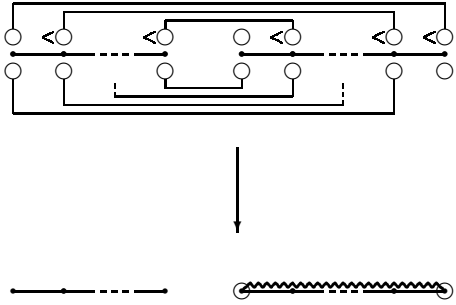
\[\begin{picture}(15900,11100)(-300,-9300)
\put(0,0){\usebox{\edge}}
\put(1800,0){\usebox{\susp}}
\multiput(0,0)(1800,0){2}{\usebox{\aone}}
\put(5400,0){\usebox{\aone}}
\put(5400,600){\usebox{\tow}}
\put(1800,600){\usebox{\tow}}
\put(8100,0){
\put(0,0){\usebox{\edge}}
\put(1800,0){\usebox{\susp}}
\multiput(0,0)(5400,0){2}{\multiput(0,0)(1800,0){2}{\usebox{\aone}}}
\multiput(5400,600)(1800,0){2}{\usebox{\tow}}
\put(1800,600){\usebox{\tow}}
\put(5400,0){\usebox{\edge}}
}
\multiput(0,900)(15300,0){2}{\line(0,1){900}}
\put(0,1800){\line(1,0){15300}}
\multiput(1800,900)(11700,0){2}{\line(0,1){600}}
\put(1800,1500){\line(1,0){11700}}
\multiput(5400,900)(4500,0){2}{\line(0,1){300}}
\put(5400,1200){\line(1,0){4500}}
\multiput(0,-900)(13500,0){2}{\line(0,-1){1200}}
\put(0,-2100){\line(1,0){13500}}
\put(1800,-900){\line(0,-1){900}}
\put(1800,-1800){\line(1,0){9900}}
\multiput(11700,-1800)(0,300){3}{\line(0,1){150}}
\multiput(3600,-1500)(0,300){2}{\line(0,1){150}}
\put(3600,-1500){\line(1,0){6300}}
\put(9900,-1500){\line(0,1){600}}
\multiput(5400,-1200)(2700,0){2}{\line(0,1){300}}
\put(5400,-1200){\line(1,0){2700}}
\put(7950,-3300){\vector(0,-1){3000}}
\put(0,-8400){
\put(0,0){\usebox{\edge}}
\put(1800,0){\usebox{\susp}}
\put(8100,0){
\put(0,0){\usebox{\edge}}
\put(1800,0){\usebox{\susp}}
\multiput(0,0)(7200,0){2}{\usebox{\wcircle}}
\put(5400,0){\usebox{\edge}}
\put(0,0){\multiput(0,0)(25,25){13}{\circle*{70}}\multiput(300,300)(300,0){22}{\multiput(0,0)(25,-25){7}{\circle*{70}}}\multiput(450,150)(300,0){22}{\multiput(0,0)(25,25){7}{\circle*{70}}}\multiput(7200,0)(-25,25){13}{\circle*{70}}}
}
}
\end{picture}\]

This is the {\it comodel} case, the pairing $\langle\rho(D),\gamma\rangle$ for $D\in\Delta$ and $\gamma\in\Sigma$ is equal to the {\it model} case, the case number 31. If $\|\Gamma\|=\{\alpha'_1\}$ or $\|\Gamma\|=\{\alpha'_p\}$ then $\mathscr P=p^2=|R^+|$. If $\|\Gamma\|=\{\alpha'_k\}$ then $\mathscr P=p^2-2(k-1)(p-k+2)$ if $1<k\leq p/2$ and $\mathscr P=p^2-2(p-k)(k+1)$ if $p/2<k<p$. 
If $\|\Gamma\|=\{\alpha'_p\}$, the maximum of $x_{\alpha'_1}$ is achieved in 
\[\theta=\sum_{i=1}^{p-1}i(i+1)\alpha_i+\sum_{i=1}^p i^2\alpha'_{p+i-1}.\]

\subsection*{50} 

$\N(\SO(p))\subset\SO(p)\times\SO(p+1)$ diagonally, $p\geq7$.

\paragraph{$\mathbf{(p=2q-1)}$:}
$\Sigma=\{\alpha_1,\ldots,\alpha_{q-1},\;\alpha'_1,\ldots,\alpha'_q\}$,
$\Delta=\{D_1,\ldots,D_p\}$, $\rho(D_{2j-1})=1,-1,-1,1$ resp.\ on $\alpha_{j-1},\alpha_{j},\alpha'_{j-1},\alpha'_j$ for all $1\leq j\leq q-2$, $\rho(D_{2j})=-1,1,1,-1$ resp.\ on $\alpha_{j-1},\alpha_{j},\alpha'_{j},\alpha'_{j+1}$ for all $1\leq j\leq q-3$, $\rho(D_{p-3})=-1,1,1,-1,-1$ resp.\ on $\alpha_{q-3},\alpha_{q-2},\alpha'_{q-2},\alpha'_{q-1},\alpha'_q$, $\rho(D_{p-2})=1,-1,-1,1,1$ resp.\ on $\alpha_{q-2},\alpha_{q-1},\alpha'_{q-2},\alpha'_{q-1},\alpha'_q$, $\rho(D_{p-1})=-1,1,1,-1$ resp.\ on $\alpha_{q-2},\alpha_{q-1},\alpha'_{q-1},\alpha'_q$ and $\rho(D_{p})=-1,1,-1,1$ resp.\ on $\alpha_{q-2},\alpha_{q-1},\alpha'_{q-1},\alpha'_q$, and $0$ on the other spherical roots, $S^p=\emptyset$, $\Delta'=\Delta\setminus\{D_1\}$, $\Sigma'=\{\alpha_1,\ldots,\alpha_{q-1}\}$.

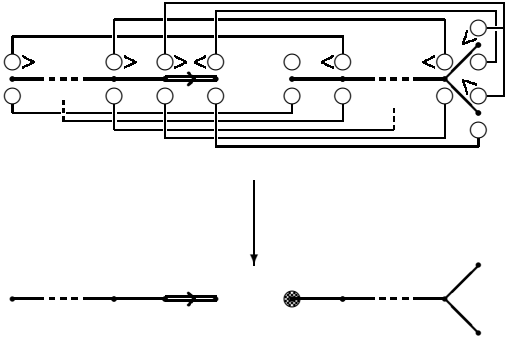
\[\begin{picture}(17700,12900)(1500,-10200)
\put(1800,0){\usebox{\susp}}
\put(1800,0){\usebox{\aone}}
\put(5400,0){\usebox{\aone}}
\put(11700,0){
\put(0,0){\usebox{\edge}}
\put(1800,0){\usebox{\susp}}
\multiput(0,0)(1800,0){2}{\usebox{\aone}}
\put(5400,0){\usebox{\aone}}
}
\put(5400,0){\usebox{\edge}}
\put(7200,0){\usebox{\rightbiedge}}
\put(17100,0){\usebox{\bifurc}}
\multiput(7200,0)(1800,0){2}{\usebox{\aone}}
\multiput(18300,-1200)(0,2400){2}{\usebox{\aone}}
\put(7200,2700){\line(0,-1){1800}}
\put(7200,2700){\line(1,0){12000}}
\put(19200,2700){\line(0,-1){3300}}
\multiput(19200,1800)(0,-2400){2}{\line(-1,0){600}}
\put(9000,2400){\line(0,-1){1500}}
\put(9000,2400){\line(1,0){9900}}
\put(18900,2400){\line(0,-1){500}}
\put(18900,1700){\line(0,-1){1100}}
\put(18900,600){\line(-1,0){300}}
\multiput(5400,2100)(11700,0){2}{\line(0,-1){1200}}
\put(5400,2100){\line(1,0){1700}}
\put(7300,2100){\line(1,0){1600}}
\put(9100,2100){\line(1,0){8000}}
\multiput(1800,1500)(11700,0){2}{\line(0,-1){600}}
\put(1800,1500){\line(1,0){3500}}
\multiput(5500,1500)(1800,0){2}{\line(1,0){1600}}
\put(9100,1500){\line(1,0){4400}}
\put(9000,-2400){\line(0,1){1500}}
\put(9000,-2400){\line(1,0){9300}}
\put(18300,-2400){\line(0,1){300}}
\multiput(7200,-2100)(9900,0){2}{\line(0,1){1200}}
\put(7200,-2100){\line(1,0){1700}}
\put(9100,-2100){\line(1,0){8000}}
\put(5400,-1800){\line(0,1){900}}
\put(5400,-1800){\line(1,0){1700}}
\put(7300,-1800){\line(1,0){1600}}
\put(9100,-1800){\line(1,0){6200}}
\multiput(15300,-1800)(0,300){3}{\line(0,1){150}}
\multiput(3600,-1500)(0,300){3}{\line(0,1){150}}
\put(3600,-1500){\line(1,0){1700}}
\multiput(5500,-1500)(1800,0){2}{\line(1,0){1600}}
\put(9100,-1500){\line(1,0){4400}}
\put(13500,-1500){\line(0,1){600}}
\multiput(1800,-1200)(9900,0){2}{\line(0,1){300}}
\put(1800,-1200){\line(1,0){1700}}
\multiput(3700,-1200)(1800,0){3}{\line(1,0){1600}}
\put(9100,-1200){\line(1,0){2600}}
\put(1800,600){\usebox{\toe}}
\put(5400,0){\multiput(0,600)(1800,0){2}{\usebox{\toe}}}
\put(9000,600){\usebox{\tow}}
\multiput(13500,600)(3600,0){2}{\usebox{\tow}}
\put(18300,1800){\usebox{\tosw}}
\put(18300,-600){\usebox{\tonw}}
\put(10350,-3600){\vector(0,-1){3000}}
\put(0,-7800){
\put(1800,0){\usebox{\susp}}
\put(11700,0){
\put(0,0){\usebox{\edge}}
\put(1800,0){\usebox{\susp}}
\put(0,0){\usebox{\gcircle}}
}
\put(5400,0){\usebox{\edge}}
\put(7200,0){\usebox{\rightbiedge}}
\put(17100,0){\usebox{\bifurc}}
}
\end{picture}\]

Let us solve the linear program \eqref{AYC} for the reduced elementary spherical skeletons with $\Gamma$ of cardinality one and support in $\Sigma\setminus\Sigma'$. Set $y_i=y_{D_i}$ for $1\leq i\leq p$ and $y_{p+1}=y_D$ where $D$ is the unique element of $\Gamma$.   

Here $|R^+|=\frac{p(p-1)}2$. If $\|\Gamma\|=\{\alpha'_1\}$, we have 
\[y_{p-2}\geq y_{p-1}+y_p+1\]
\[y_{p-4}\geq y_{p-3}+y_{p-2}-y_{p-1}-y_p-1\geq y_{p-3}\]
and (recursively) for all $1\leq j\leq q-3$
\[y_{2j-1}\geq y_{2j}+y_{2j+1}-y_{2j+2}\geq y_{2j}.\]
On the other hand, we have
\[y_{p-3}\geq y_{p-2}\]
and (recursively) for all $1\leq j\leq q-3$
\[y_{2j}\geq y_{2j+1}+y_{2j+2}-y_{2j+3}\geq y_{2j+1},\]
as well as
\[y_{p+1}\geq y_{1}+y_{2}-y_{3}\geq y_{1}.\]
Then $\mathscr P=p-1$.

If $\|\Gamma\|=\{\alpha'_{k}\}$ for $2\leq k\leq q-2$, we have as above
\[y_{p-2}\geq y_{p-1}+y_p+1\]
and for all $1\leq j\leq q-2$
\[y_{2j-1}\geq  y_{2j},\]
and similarly for $k \leq j\leq q-2$
\[y_{2j}\geq y_{2j+1}.\]
On the other hand, we have
\[y_4\geq y_3+y_2-y_1\geq 2y_2+1\]
and (recursively) for all $3\leq j\leq k$
\[y_{2j}\geq y_{2j-1}+y_{2}-y_{1}\geq y_{2j-2}+y_2+1\geq j\,y_2+(j-1),\] 
as well as
\[y_{2k+1}+y_{p+1}\geq y_{2k}+y_{2k-1}-y_{2k-2}-1\geq y_{2k}+y_1.\]
Then $\mathscr P=p+(k-1)^2-5$.

If $\|\Gamma\|=\{\alpha'_{q}\}$ (the case $\alpha'_{q-1}$ is equivalent by symmetry), we have again
\[y_{p-2}\geq y_{p-1}+y_p+1\]
and for all $1\leq j\leq q-2$
\[y_{2j-1}\geq  y_{2j}.\]
On the other hand, we have
\[y_4\geq y_3+y_2-y_1\geq 2y_2+1\]
and (recursively) for all $3\leq j\leq q-2$
\[y_{2j}\geq y_{2j-1}+y_{2}-y_{1}\geq y_{2j-2}+y_2+1\geq j\,y_2+(j-1),\] 
as well as
\[y_{2q-3}\geq y_{2q-4}+y_{1}+1\geq (q-2)y_{2}+y_1+(q-2),\]
\[2\,y_{2q-2}+y_{2q}\geq 2\,y_{2q-3}-y_{2q-5}-2\geq y_{2q-4}+y_2+y_1\geq (q-1)y_2+y_1+(q-3),\]
\[2\,y_{2q-1}\geq y_{2q-3}+y_2\geq (q-1)y_2+y_1+(q-2),\]
\[y_{2q}\geq 2\,y_1+1.\]
Then $\mathscr P=(p-1)(p-3)/4$.

\paragraph{$\mathbf{(p=2q)}$:}
$\Sigma=\{\alpha_1,\ldots,\alpha_{q},\;\alpha'_1,\ldots,\alpha'_q\}$,
$\Delta=\{D_1,\ldots,D_p\}$, $\rho(D_{2j-1})=-1,1,1,-1$ resp.\ on $\alpha_{j-1},\alpha_{j},\alpha'_{j-1},\alpha'_{j}$ for all $1\leq j\leq q-2$, $\rho(D_{2j})=1,-1,-1,1$ resp.\ on $\alpha_{j},\alpha_{j+1},\alpha'_{j-1},\alpha'_{j}$ for all $1\leq j\leq q-2$, $\rho(D_{p-3})=-1,1,1,-1,-1$ resp.\ on $\alpha_{q-2},\alpha_{q-1},\alpha'_{q-2},\alpha'_{q-1},\alpha'_q$, $\rho(D_{p-2})=1,-1,-1,1,1$ resp.\ on $\alpha_{q-1},\alpha_{q},\alpha'_{q-2},\alpha'_{q-1},\alpha'_q$, $\rho(D_{p-1})=-1,1,-1,1$ resp.\ on $\alpha_{q-1},\alpha_{q},\alpha'_{q-1},\alpha'_q$ and $\rho(D_{p})=-1,1,1,-1$ resp.\ on $\alpha_{q-1},\alpha_{q},\alpha'_{q-1},\alpha'_q$, and $0$ on the other spherical roots, $S^p=\emptyset$, $\Delta'=\Delta\setminus\{D_1\}$, $\Sigma'=\{\alpha'_1,\ldots,\alpha'_{q}\}$.

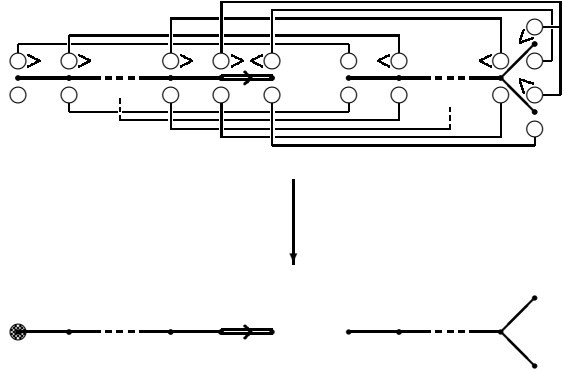
\[\begin{picture}(19500,12900)(-300,-10200)
\multiput(0,0)(11700,0){2}{
\put(0,0){\usebox{\edge}}
\put(1800,0){\usebox{\susp}}
\multiput(0,0)(1800,0){2}{\usebox{\aone}}
\put(5400,0){\usebox{\aone}}
}
\put(5400,0){\usebox{\edge}}
\put(7200,0){\usebox{\rightbiedge}}
\put(17100,0){\usebox{\bifurc}}
\multiput(7200,0)(1800,0){2}{\usebox{\aone}}
\multiput(18300,-1200)(0,2400){2}{\usebox{\aone}}
\put(7200,2700){\line(0,-1){1800}}
\put(7200,2700){\line(1,0){12000}}
\put(19200,2700){\line(0,-1){3300}}
\multiput(19200,1800)(0,-2400){2}{\line(-1,0){600}}
\put(9000,2400){\line(0,-1){1500}}
\put(9000,2400){\line(1,0){9900}}
\put(18900,2400){\line(0,-1){500}}
\put(18900,1700){\line(0,-1){1100}}
\put(18900,600){\line(-1,0){300}}
\multiput(5400,2100)(11700,0){2}{\line(0,-1){1200}}
\put(5400,2100){\line(1,0){1700}}
\put(7300,2100){\line(1,0){1600}}
\put(9100,2100){\line(1,0){8000}}
\multiput(1800,1500)(11700,0){2}{\line(0,-1){600}}
\put(1800,1500){\line(1,0){3500}}
\multiput(5500,1500)(1800,0){2}{\line(1,0){1600}}
\put(9100,1500){\line(1,0){4400}}
\multiput(0,1200)(11700,0){2}{\line(0,-1){300}}
\put(0,1200){\line(1,0){1700}}
\put(1900,1200){\line(1,0){3400}}
\multiput(5500,1200)(1800,0){2}{\line(1,0){1600}}
\put(9100,1200){\line(1,0){2600}}
\put(9000,-2400){\line(0,1){1500}}
\put(9000,-2400){\line(1,0){9300}}
\put(18300,-2400){\line(0,1){300}}
\multiput(7200,-2100)(9900,0){2}{\line(0,1){1200}}
\put(7200,-2100){\line(1,0){1700}}
\put(9100,-2100){\line(1,0){8000}}
\put(5400,-1800){\line(0,1){900}}
\put(5400,-1800){\line(1,0){1700}}
\put(7300,-1800){\line(1,0){1600}}
\put(9100,-1800){\line(1,0){6200}}
\multiput(15300,-1800)(0,300){3}{\line(0,1){150}}
\multiput(3600,-1500)(0,300){3}{\line(0,1){150}}
\put(3600,-1500){\line(1,0){1700}}
\multiput(5500,-1500)(1800,0){2}{\line(1,0){1600}}
\put(9100,-1500){\line(1,0){4400}}
\put(13500,-1500){\line(0,1){600}}
\multiput(1800,-1200)(9900,0){2}{\line(0,1){300}}
\put(1800,-1200){\line(1,0){1700}}
\multiput(3700,-1200)(1800,0){3}{\line(1,0){1600}}
\put(9100,-1200){\line(1,0){2600}}
\multiput(0,0)(5400,0){2}{\multiput(0,600)(1800,0){2}{\usebox{\toe}}}
\put(9000,600){\usebox{\tow}}
\multiput(13500,600)(3600,0){2}{\usebox{\tow}}
\put(18300,1800){\usebox{\tosw}}
\put(18300,-600){\usebox{\tonw}}
\put(9750,-3600){\vector(0,-1){3000}}
\put(0,-9000){
\multiput(0,0)(11700,0){2}{
\put(0,0){\usebox{\edge}}
\put(1800,0){\usebox{\susp}}
}
\put(0,0){\usebox{\gcircle}}
\put(5400,0){\usebox{\edge}}
\put(7200,0){\usebox{\rightbiedge}}
\put(17100,0){\usebox{\bifurc}}
}
\end{picture}\]

Here $|R^+|=\frac{p(p-1)}2$. If $\|\Gamma\|=\{\alpha_1\}$, we have 
\[y_{p-3}\geq y_{p-2}\]
and (recursively) for all $1\leq j\leq q-2$
\[y_{2j-1}\geq y_{2j}+y_{2j+1}-y_{2j+2}\geq y_{2j}.\]
On the other hand, we have
\[y_{p-2}\geq y_{p-1}+y_p+1,\]
\[y_{p-4}\geq y_{p-3}+y_{p-2}-y_{p-1}-y_p-1\geq y_{p-3}\]
and (recursively) for all $1\leq j\leq q-3$
\[y_{2j}\geq y_{2j+1}+y_{2j+2}-y_{2j+3}\geq y_{2j+1},\]
as well as
\[y_{p+1}\geq y_{1}+y_{2}-y_{3}\geq y_{1}.\]
Then $\mathscr P=p-1$.

If $\|\Gamma\|=\{\alpha_{k}\}$ for $2\leq k\leq q-2$, we have as above
for all $1\leq j\leq q-1$
\[y_{2j-1}\geq y_{2j},\]
and similarly 
\[y_{p-2}\geq y_{p-1}+y_p+1,\]
and for $k\leq j\leq q-2$
\[y_{2j}\geq y_{2j+1}.\]
On the other hand, we have
\[y_4\geq y_3+y_2-y_1\geq 2y_2+1\]
and (recursively) for all $3\leq j\leq k$
\[y_{2j}\geq y_{2j-1}+y_{2}-y_{1}\geq y_{2j-2}+y_2+1\geq j\,y_2+(j-1),\] 
as well as
\[y_{2k+1}+y_{p+1}\geq y_{2k}+y_{1}\geq k\,y_{2}+y_1+(i-1).\]
Then $\mathscr P=p+(k-1)^2-5$.

If $\|\Gamma\|=\{\alpha_{q-1}\}$, similarly we have 
for all $1\leq j\leq q-1$
\[y_{2j-1}\geq y_{2j},\]
\[y_{p-2}\geq y_{p-1}+y_p+1,\]
and for all $2\leq j\leq q-1$
\[y_{2j}\geq j\,y_2+(j-1),\] 
as well as
\[y_{p-1}+y_p+y_{p+1}\geq y_{p-2}+y_{1}\geq (q-1)y_{2}+y_1+(q-2).\]
Then $\mathscr P=p+(q-2)^2-5=(p^2-4p-4)/4$.

If $\|\Gamma\|=\{\alpha_{q}\}$, we have again
for all $1\leq j\leq q-1$
\[y_{2j-1}\geq y_{2j},\]
and for all $1\leq j\leq q-2$
\[y_{2j+1}\geq  y_{2j}+y_1+1,\]
\[y_{p-1}+y_p\geq y_{p-2}+y_1,\]
\[y_{p+1}\geq y_1,\]
as well as for all $1\leq j\leq q-1$
\[y_{2j}\geq j\,y_2+(j-1),\] 
and
\[y_{p-1}+y_p\geq (q-1)y_{2}+y_1+(q-1).\]
Then $\mathscr P=p(p-4)/4$.

\subsection*{The cases where equality holds}

From the previous computations, the list of the reduced elementary complete spherical skeletons $\mathcal R=(\Delta,S^p,\Sigma,\Gamma)$, with $\Gamma$ of minimal support and underlying non-symmetric indecomposable spherically closed spherical system corresponding to a reductive spherical subgroup, with $\mathscr P(\mathcal R)=|R^+\setminus R_{S^p}^+|$, is the following: 

\begin{itemize}

\item[-] case 31, $\|\Gamma\|=\{\gamma_1\}$ or $\|\Gamma\|=\{\gamma_{2p-1}\}$,

\item[-] case 32 for $p=2$, $\|\Gamma\|=\{\gamma_1\}$,

\item[-] case 34, $\|\Gamma\|=\{\gamma_1\}$,

\item[-] case 35, $\|\Gamma\|=\{\gamma\}$,

\item[-] case 36, $\|\Gamma\|=\{\gamma_2\}$,

\item[-] case 38, any $\|\Gamma\|\subset\{\gamma_1,\gamma_2,\gamma_3\}$ of cardinality 2,

\item[-] case 41, $\|\Gamma\|=\{\gamma\}$,

\item[-] case 42 for $p=0$, $\|\Gamma\|=\{\gamma_2\}$,

\item[-] case 43 for $p=q=r=0$, any $\|\Gamma\|\subset\{\alpha_1,\alpha'_1,\alpha''_1\}$ of cardinality 2,

\item[-] case 43 for $p\neq0$ and $q=r=0$, $\|\Gamma\|=\{\gamma_1,\gamma_4\}$ or $\|\Gamma\|=\{\gamma_2,\gamma_4\}$,

\item[-] case 43 for $p,q\neq0$ and $r=0$, $\|\Gamma\|=\{\gamma_3,\gamma_5\}$,

\item[-] case 46 for $p=5$, $\|\Gamma\|=\{\alpha'_1\}$ or $\|\Gamma\|=\{\alpha'_3\}$,

\item[-] case 49, $\|\Gamma\|=\{\alpha'_1\}$ or $\|\Gamma\|=\{\alpha'_p\}$.

\end{itemize}

These are all spherical skeletons of spherical modules, indeed, they can all be found in \cite[List L]{G} with numbers, respectively: 24, 38 for $n=2$, 16, 15, 38 for $n>2$, 20, 18, 10, 35, 40, 42, 13, 28. Notice also that the point $\theta$, exibited case by case in the above computations, is the unique point where the maximum of the functional is attained by Corollary~\ref{cor:1.2mfs}.

To complete the proof of  Theorem~\ref{MT} we adapt the argument of \cite[Section~10]{GH}. Let $\mathcal R=(\Delta,S^p,\Sigma,\Gamma)$ be a complete spherical skeleton with underlying spherically closed spherical system $\mathcal S=(\Delta,S^p,\Sigma)$ corresponding to a reductive spherical subgroup with $\mathscr P(\mathcal R)=|R^+\setminus R_{S^p}^+|$. We may assume that, for all $D\in\Gamma$, $\rho(D)$ is not identically zero on $\Sigma$. Let us consider $\mathcal R^r=(\Delta,S^p,\Sigma,\Gamma^r)$, we necessarily have $\mathscr P(\mathcal R)=\mathscr P(\mathcal R^r)=|R^+\setminus R_{S^p}^+|$. The reduced elementary spherical skeleton $\mathcal R^r$ decomposes as $\mathcal R^r=\mathcal R_1^r\times\cdots\times\mathcal R_k^r$ with underlying spherically closed spherical system $\mathcal S$ which correspondingly decomposes as $\mathcal S=\mathcal S_1\times\cdots\times\mathcal S_k$ where, for all $1\leq i\leq k$, $\mathcal S_i$ belongs to the list of \cite[Section 3]{BP15} and $\Gamma^r=\Gamma_1^r\sqcup\ldots\sqcup\Gamma_k^r$. For all $1\leq i\leq k$, $\mathcal R_i^r$ is a complete spherical skeleton satisfying the equality in \eqref{ineq}, $\mathcal P(\mathcal R_i^r)=|R_i^+\setminus R_{i,S_i^p}^+|$.

We claim that the support of $\Gamma_i^r$ is minimal, for all $1\leq i\leq k$. Indeed, every subset of $\Gamma_i^r$ of minimal support must be one of the above list (or symmetric), since equality must still hold, so we are left to exclude the following non-symmetric cases where $\mathscr P(\mathcal R_i^r)$ is strictly less than $|R_i^+\setminus R_{i,S_i^p}^+|$:

\begin{itemize}

\item[-] case 31, $\|\Gamma_i^r\|=\{\gamma_1,\gamma_{2p-1}\}$, $\mathscr P=2p<2p^2+p$,

\item[-] case 38, $\|\Gamma_i^r\|=\{\gamma_1,\gamma_2,\gamma_3\}$, $\mathscr P=6<11$,

\item[-] case 43 for $p=q=r=0$, any $\|\Gamma_i^r\|=\{\alpha_1,\alpha'_1,\alpha''_1\}$, $\mathscr P=0<3$,

\item[-] case 43 for $p\neq0$ and $q=r=0$, $\|\Gamma_i^r\|=\{\gamma_1,\gamma_2,\gamma_4\}$, $\mathscr P=2p-1<4p+2$

\item[-] case 46 for $p=5$, $\|\Gamma_i^r\|=\{\alpha'_1,\alpha'_3\}$, $\mathscr P=1<10$,

\item[-] case 49, $\|\Gamma_i^r\|=\{\alpha'_1,\alpha'_p\}$, $\mathscr P=0<p^2$.

\end{itemize}

Therefore, every $\mathcal R_i^r$ is the spherical skeleton of a spherical module, and by the way also $\mathcal R^r$ is.

We claim that $\mathcal R^e=\mathcal R^r$, that is, for every $\gamma\in\Sigma$, $\sum_{D\in\Gamma}\langle\rho(D),\gamma\rangle=-1$. Indeed, we can assume that $\mathcal R^r$ is indecomposable, hence one of the above list of cases (or symmetric). If we construct a new spherical skeleton $\mathcal R'$ by duplicating one of the elements $D\in\Gamma^r$, we keep the same polytope $\mathcal Q^*_{\mathcal R'}=\mathcal Q^*_{\mathcal R^r}$ and $\mathscr P(\mathcal R')=\mathscr P(\mathcal R^r)+\langle\rho(D),\theta\rangle$, where $\theta$ is the vertex of $\mathcal Q^*_{\mathcal R^r}\cap\mathrm{cone}(\Sigma)$ where the maximum $\mathscr P(\mathcal R^r)$ is achieved. We get $\mathscr P(\mathcal R')$ strictly smaller than $\mathscr P(\mathcal R^r)$, since in all cases the support of $\theta$ contains $\|\Gamma^r\|$. Actually, in all cases the support of $\theta$ equals the whole $\Sigma$, thanks to Lemma~\ref{lemma:interior}.

We claim that $\mathcal R$ is elementary, hence reduced. If $\mathcal R\neq\mathcal R^e=\mathcal R^r$ we can reconstruct the elements of $\Gamma$ by recombining the elements of $\Gamma^e$ step by step, in particular we can start by combining together only two elements in $\Gamma^e$ and leaving unchanged all the others. Therefore, we can assume that $\mathcal R^e$ is either an indecomposable spherical skeleton of the above list or the product of two indecomposable spherical skeletons $\mathcal R_1^e\times\mathcal R_2^e$ of the above list (or symmetric) and the two combining elements belong to $\Gamma_1^e$ and $\Gamma_2^e$, respectively. In the former case, we have the following possibilities, with $\Gamma$ of cardinality 1 and $\|\Gamma\|$ of cardinality 2, where $\mathscr P(\mathcal R)$ is always strictly less than $|R^+\setminus R_{S^p}^+|$:   

\begin{itemize}

\item[-] case 38, $\mathscr P=10<11$,

\item[-] case 43 for $p=q=r=0$, $\mathscr P=2<3$,

\item[-] case 43 for $p\neq0$ and $q=r=0$, $\mathscr P=4p+1<4p+2$,

\item[-] case 43 for $p,q\neq0$ and $r=0$, $\mathscr P=4p+4q<4p+4q+1$.

\end{itemize}

In the latter case, as in \cite[Lemma 10.8]{GH}, $\mathscr P(\mathcal R_i^e)$ achieves its optimal value only in $\theta_i$ but $(\theta_1,\theta_2)\not\in\mathcal Q^*_{\mathcal R}\cap\mathrm{cone}(\Sigma)$, since in all cases
\begin{equation}
\mbox{\textit{$x_{\gamma}(\theta_i)=1$ if $\gamma\in\|\Gamma_i^e\|$.}}
\end{equation}  

Therefore, $\mathcal R$ is the spherical skeleton of a spherical module.

\end{document}